\documentclass[a4paper,11pt,makeidx]{amsart}
\usepackage{amscd}
\usepackage{xypic}  
\usepackage{amssymb}
\usepackage{amsthm}
\usepackage{epsf}
\makeindex
\newtheorem{thm}{Theorem}[section]
\newtheorem{cor}[thm]{Corollary}

\newtheorem{lemma}[thm]{Lemma}

\newtheorem{prop}[thm]{Proposition}

\newtheorem{defn}[thm]{Definition}
\newtheorem{claim}[thm]{Claim}
\newtheorem{rem}[thm]{Remark}

\renewcommand{\proofname}{Proof}

\def\Sing{\operatorname{Sing}}

\def\codim{\operatorname{codim}}
\def\min{\operatorname{min}}
\def\Im{\operatorname{Im}}
\def\rk{\operatorname{rk}}

\def\max{\operatorname{max}}
\def\length{\operatorname{length}}
\def\c1{\operatorname{c_1}}
\def\c2{\operatorname{c_2}}
\def\Cliff{\operatorname{Cliff}}

\def\gon{\operatorname{gon}}

\def\QQ{{\mathbb Q}}
\def\PP{{\mathbb P}}

\def\L{{\mathcal L}}

\def\R{{\mathcal R}}
\def\N{{\mathcal N}}
\def\O{{\mathcal O}}
\def\I{{\mathcal J}}
\def\G{{\mathcal G}}

\def\E{{\mathcal E}}

\def\F{{\mathcal F}}
\def\K{{\mathcal K}}

\def\J{{\mathfrak J}}
\def\*{\otimes}
\def\x{\times}                  
\def\iso{\cong}
\def\eqv{\equiv}
\def\sub{\subseteq}

\def\sup{\supseteq}

\def\+{\oplus}                   
\def\*{\otimes}                  
\def\hpil{\longrightarrow}       
\def\khpil{\rightarrow}

\def\Pic{\operatorname{Pic}}

\def\Supp{\operatorname{Supp}}

\def\det{\operatorname{det}}
\def\Bs{\operatorname{Bs}}
\def\hs{\hspace{.1in}}

\begin{document}

\title[Brill-Noether theory of curves on Enriques surfaces, II. The Clifford index]{Brill-Noether theory
of curves on Enriques surfaces, II. The Clifford index}

\author[A.L. Knutsen and A.F. Lopez]{Andreas Leopold Knutsen* and Angelo Felice Lopez**}

\address{\hskip -.43cm Andreas Leopold Knutsen, Department of Mathematics, University of Bergen,
Johannes Brunsgate 12, 5008 Bergen, Norway. e-mail {\tt andreas.knutsen@math.uib.no}}

\thanks{* Research partially supported by a Marie Curie Intra-European Fellowship within the 6th European
Community Framework Programme}

\address{\hskip -.43cm Angelo Felice Lopez, Dipartimento di Matematica, Universit\`a di Roma 
Tre, Largo San Leonardo Murialdo 1, 00146, Roma, Italy. e-mail {\tt lopez@mat.uniroma3.it}}

\thanks{** Research partially supported by the MIUR national project ``Geometria delle variet\`a algebriche e dei loro spazi di moduli".}


\begin{abstract}
We complete our study of linear series on curves lying on an Enriques surface
by showing that, with the exception of smooth plane quintics, there are no 
exceptional curves on Enriques surfaces, that is, curves for which the Clifford index is not computed by a pencil.
\end{abstract}

\maketitle

\section{Introduction}
\label{intro}

On a smooth irreducible curve $C$ of genus at least $4$ there are two very important and much studied invariants, a classical one, the gonality, $\gon C$, and a modern one, the Clifford index, $\Cliff C$. Their importance ranges from projective geometrical to moduli properties and tells a lot about the curve itself, for example, when $C$ is not hyperelliptic, about the syzygies of its ideal in the canonical embedding.

After the work of Coppens and Martens \cite{cm} we know that there is a relation between these invariants
\[ \gon C -3 \leq \Cliff C \leq \gon C -2 \]
and one would like to know what are the properties of curves realizing one of the two equalities. As it turns out, for the general curve one has $\Cliff C = \gon C -2$, while curves for which $\Cliff C=\gon C -3$, called {\it exceptional curves}, are conjectured to be extremely rare \cite{elms}. As a matter of fact, aside for smooth plane curves, very few cases of exceptional curves are known, almost all lying on $K3$ surfaces \cite{elms}.

The starting idea of this work was that, given the flexibility and richness of the Picard group of Enriques surfaces, we should investigate if there are exceptional curves lying on them. One such case was already known, the one of smooth plane quintics \cite{sta, umeq}.

The main result of this note is that, in fact, the above are the only examples:

\begin{thm} 
\label{thm:noexc}
On an Enriques surface there are no exceptional curves other than smooth plane quintics.

In particular, for any smooth curve $C$ on an Enriques surface $S$ such that $C^2 \neq 10$, we have $\Cliff C = \gon C-2$.
\end{thm}

This result gives more evidence for the conjecture in \cite{elms}.

We remark that similar results were proved for curves on del Pezzo and $K3$ surfaces by the first author in \cite{kn2} and \cite{kn3}.

Now in \cite{kl1} we computed the gonality of a general smooth curve $C$ in a linear system $|L|$ on an Enriques surface $S$. Recalling the two functions \cite{cd}, \cite[Def.1.1 and 1.2]{kl1}
\[ \phi(L) :=\inf \{|F.L| \; : \; F \in \Pic S, F^2 = 0, F \not\eqv 0\}\] 
\[\mu(L) = \min\{B.L - 2 \ : \ B \in \Pic(S) \ \hbox{with} \ B^2 = 4, \phi(B) = 2, B \not\eqv L \}, \] 
as an immediate consequence of Theorem \ref{thm:noexc} and \cite[Thm.1.3 and Prop.4.13]{kl1}, we are now able to compute the Clifford index of $C$:

\begin{cor} 
\label{cliff}
Let $|L|$ be a base-component free linear system on an Enriques surface with $L^2 \geq 6$ and let $C$ be a general curve in $|L|$. Then
\[  \Cliff C = \min\{2 \phi(L)-2, \mu(L)-2, \lfloor \frac{L^2}{4} \rfloor  \}. \]
\end{cor}
As a matter of fact, by Theorem \ref{thm:noexc} and \cite[Cor.1.5]{kl1}, the cases when the Clifford index is not $2 \phi(L)-2$ are completely characterized.

We point out that this is particularly important for us in the study of Gaussian maps on curves on Enriques surfaces in \cite{klGM}, which is a key ingredient 
to obtain the genus bound $g \leq 17$ for Enriques-Fano threefolds in \cite{klm}. In fact, the results in \cite{klGM} depend on the Clifford index of the curves and not on their gonality.

Another application of Theorem \ref{thm:noexc} will be in \cite{kl4}, where we will prove that a linearly normal Enriques surface $S \subset \PP^r$ is scheme-theoretically cut out by quadrics if and only if $\phi(\O_S(1)) \geq 4$ (improving \cite[Thm.1.3]{glm2}) and that, when $\phi(\O_S(1)) = 3$ and $\deg S \geq 18$, the intersection of the quadrics containing $S$ is the union of $S$ and the 2-planes spanned by the plane cubics contained in $S$. Moreover, in \cite{kl4} we will also use Theorem \ref{thm:noexc} to give a new proof (after \cite[Thm.1.1]{glm1}) of the projective normality of a linearly normal Enriques surface $S \subset \PP^r$ of degree at least $12$.

We now give an outline of the ideas concurring in the proof of our main result. The {\it Clifford dimension} of a smooth
curve is defined to be the least integer $r$ such that there is a $g^r_d$ computing its Clifford index. In this language, the exceptional curves are precisely the ones of Clifford dimension at least $2$ and curves of Clifford dimensions $r \leq 9$ are well classified by \cite{elms} and \cite{ma}, for example for $r=2$ we get smooth plane curves, for $r=3$ complete intersections of two cubics in $\PP^3$. In general the study of $g^r_d$'s with $r \geq 2$ on curves on surfaces by using vector bundle methods is much harder than the case $r=1$, because the vector bundles arising have ranks at least $3$, and are therefore much more difficult to handle than the ones of rank two, where various instability criteria can be used. In this note we show how to overcome this difficulty on an Enriques surface, but our methods and ideas can in principle be used also on other surfaces. The main idea is to use the geometry of the surface to find suitable line subbundles of the vector bundles, and after saturating we study the quotient bundle, which is of rank one less. This seems to be a promising method in the cases where one knows that the Picard group of the surface is particularly ``rich'' (at least of rank two!) However, as we will see below, vector bundle methods are not sufficient to treat the cases of Clifford dimension $3$, where we will need a more geometric approach, see Section \ref{sec:r=3}.

We give some preliminary results in Section \ref{sec:preliminari}. Then we divide the proof of Theorem \ref{thm:noexc} into four parts: the cases of plane curves ($r=2$) in Section \ref{sec:plane}, the cases of the complete intersections of two cubics ($r=3$) in Section \ref{sec:r=3}, the cases of Clifford dimension from $4$ to $9$ in Section \ref{sec:r=4-9}, and the cases of higher Clifford dimension in Section \ref{sec:r>9}. Finally in Section  \ref{sec:w14} we prove a result about Brill-Noether loci announced in \cite[Rmk.4.16]{kl1}.

\section{Preliminary results}
\label{sec:preliminari}

In this section we will gather some results that will be used throughout the note.

\begin{defn}
We denote by $\sim$ (resp. $\eqv$) the linear (resp. numerical) equivalence of divisors or line bundles on a smooth surface. A line bundle $L$ is {\bf primitive} if $L \eqv kL'$ implies $k = \pm 1$. If $V \subseteq H^0(L)$ is a linear system, we denote its {\bf base scheme} by $\Bs |V|$. A {\bf nodal} curve on an Enriques surface $S$ is a smooth rational curve contained in $S$. A {\bf nodal cycle} is a divisor $R>0$ such that, for any $0 < R' \leq R$ we have $(R')^2 \leq -2$.
\end{defn} 

We will use that if $R$ is a nodal cycle, then $h^0(\O_S(R)) = 1$ and $h^0(\O_S(R+K_S)) = 0$. 

\begin{lemma}
Let $C$ be an exceptional curve of Clifford dimension $r$ on an Enriques surface. Then
\begin{equation}
\label{eq:E2a}
\phi(C) \geq r
\end{equation}
and 
\begin{equation}
\label{eq:E2b}
\Cliff C \leq 2\phi(C) -3.
\end{equation}
\end{lemma}
 
\begin{proof} By \cite[Proof of Prop.3.2]{elms} any (not necessarily complete) pencil of divisors on $C$ has degree $\geq 2r$, whence \eqref{eq:E2a}. Since $\Cliff C = \gon C-3$ and $\gon C \leq 2\phi(C)$ we get  \eqref{eq:E2b}.
\end{proof}

\begin{lemma}
\label{cliffhanger}
Let $L$ be a base-point free line bundle on an Enriques surface $S$ with $L^2 \geq 6$. Assume that $L + K_S \sim D_1+ D_2$ for two divisors $D_1$ and $D_2$ satisfying $h^0(D_i) \geq 2$, $i = 1, 2$. Then $\O_C(D_1)$ and $\O_C(D_2)$ contribute to the Clifford index of any smooth $C \in |L|$ and
\[ \Cliff \O_C(D_1) = \Cliff \O_C(D_2) \leq D_1.D_2 -2 \max \{h^1(D_1), h^1(D_2) \} \leq  D_1.D_2.   \]
Also $\Cliff \O_C(D_1) = D_1.D_2$ if and only if $h^0(\O_C(D_i)) = h^0(D_i)$ and $h^1(D_i)=0$ for $i = 1, 2$.
\end{lemma}

\begin{proof} This follows the lines and ideas in \cite[Lemma3.6]{kn2}.
\end{proof}
Given a smooth curve $C$ on a smooth surface $S$ and a base-point free line bundle $A$ on $C$, a standard construction (\cite{cp}, \cite{la1}, \cite{par}) allows to define a vector bundle $\E(C,A)$ of rank $h^0(A)$ and with $\det \E(C,A) = \O_S(C)$, sitting in an exact sequence
\begin{equation} 
\label{eq:eca}
0 \hpil H^0(A)^{\ast} \* \O_S \hpil \E \hpil \N_{C/S}-A \hpil 0
\end{equation}
and whose properties are listed in the mentioned references.
\begin{lemma}
\label{useful3}
Let $C$ be a smooth irreducible curve on a smooth irreducible regular surface $S$ and let $A$ be a complete base-point free $g^r_d$ on $C$ with $r \geq 1$ and $h^0(\N_{C/S} - A) > 0$. Let $s \in H^0(\E(C,A))$ be a nonzero section and let $D \geq 0$ be the divisorial subscheme of the zero locus of $s$. Then we have an exact sequence
\begin{equation}
\label{eq:E8}
0 \hpil \O_S(D) \hpil \E(C,A) \hpil \F \hpil \tau \hpil 0,
\end{equation}
where $\F$ is locally free of rank $r$, $\tau$ is a torsion sheaf supported on a finite set and $\F$ is globally generated off a finite set contained in $C \cup \Supp(\tau)$. Moreover $h^0(\F^*) = 0$ and, if $D > 0$, also $h^1(\F^*) = 0$.
Define $M = \det \F$. Then $M$ is nontrivial, base-component free, $C \sim M + D$ and, if $D > 0$, then $M_{|C} \geq A$. Finally
\begin{equation}
\label{eq:E11}
\Cliff A = D.M + \length(\tau) + c_2(\F) - 2 \rk \F.
\end{equation}
\end{lemma}
\begin{proof}
The exact sequence (\ref{eq:E8}) and the facts that $\F$ is locally free of rank $r$ and $\tau$ is a torsion sheaf supported on a finite set are standard (\cite[2.12]{gl1} or \cite[1.11]{par}). As $h^0(\N_{C/S} - A) > 0$ we see from \eqref{eq:eca} that $\E = \E(C,A)$ is globally generated off a finite set contained in $C$, whence $\F$ is globally generated off a finite set contained in $C \cup \Supp(\tau)$. As $H^i(\E(K_S)) = 0$ for $i = 1, 2$ we get from \eqref{eq:E8} that $h^0 (\F^*) = 0$ and, if $D > 0$, also $h^1(\F^*) = 0$. Taking $c_1$ in (\ref{eq:E8}) yields $C \sim M+D$ and using $c_2(\E) =  \deg A$ and (\ref{eq:E8}) we get
\[ \Cliff A =  c_2(\E) -2(\rk \E-1) =  D.M + c_2 (\F) + \length(\tau) - 2\rk \F, \]
showing (\ref{eq:E11}). If $D > 0$, tensoring \eqref{eq:eca} and (\ref{eq:E8}) by $\O_S(-D)$ and taking global sections yields $h^0(M_{|C}-A) > 0$, that is $M_{|C} \geq A$. In particular we get that $M$ is nontrivial, otherwise we would have that $D \sim C > 0$, hence $\O_C \geq A$, a contradiction. Moreover $M$ is globally generated off a finite set since $\F$ is, hence $M$ is base-component free.
\end{proof}
\begin{lemma}
\label{bigger0}
Let $C$ be a smooth irreducible curve of Clifford index $c$ on an Enriques surface $S$, let $L = \O_S(C)$ and let $A$ be a line bundle on $C$ that computes the Clifford index of $C$ with $h^0(\N_{C/S} - A) > 0$ and $h^0(A) \geq 3$. Let $s \in H^0(\E(C,A))$ be a nonzero section and let $D \geq 0$ be the divisorial subscheme of the zero locus of $s$ and let $\F, M$ be defined as in Lemma {\rm \ref{useful3}}. Suppose that $L^2 \geq 6, c \leq 2 \phi(L) - 3$ and $h^0(D) \geq 2$.

Then $D^2 > 0$ and $M^2 > 0$. Consequently $h^1(M) = h^1(M + K_S) = 0$, $h^1(D) \leq 1$, $h^1(D + K_S) \leq 1$, $-2 \leq c_2(\F) - 2 \rk \F \leq 0$ and $D.M \leq c + 2$.
\end{lemma}
\begin{proof}
Assume first that $M^2 = 0$. Then $M \sim mP$ for an elliptic pencil $|P|$ and an integer $m \geq 1$. By \cite[Prop.3.2]{kn2} we have $c_2 (\F) - 2\rk \F \geq -2m$, whence by (\ref{eq:E11}) we get
\[c \geq mP.D - 2m = mP.L - 2m \geq 2m\phi(L) - 2m \geq 2\phi(L) - 2,\]
a contradiction. Therefore $M^2 > 0$, hence $h^1(M) = h^1(M + K_S) = 0$ since $M$ is nef, and by \cite[Prop.3.2]{kn2} again it follows that $c_2(\F) - 2\rk \F \geq - 2$. It follows from (\ref{eq:E11}) again that $D.M \leq c + 2$. Note that $\phi(L) \geq 2$, whence $L$ is base-point free by \cite[Thm.4.4.1]{cd}.

Since $M^2 \geq 2$ we have $h^0(M + K_S) \geq 2$, whence by Lemma \ref{cliffhanger} we must have $D.M \geq c + 2h^1(D)$, hence $h^1(D) \leq 1$. The latter immediately yields by Riemann-Roch that $D^2 \geq 0$, and if $D^2 = 0$, we get $2 \phi(L) \leq D.L = D.M \leq c + 2$, contradicting our assumption. It follows that $D^2 >0$, whence $h^0(D + K_S) \geq 2$, $h^1(D + K_S) \leq 1$ and $D.M \geq c$ by Lemma \ref{cliffhanger} once more. Now $D.M \geq c$ implies $c_2 (\F) - 2\rk \F \leq 0$.
\end{proof}
Let $\R$ be a vector bundle of rank $r \geq 1$ on a smooth surface $S$ with $\R$ globally generated off a finite set, $h^0(\R^*)=0$ and $c_1(\R)^2 >0$. It is a standard fact (see for example the proof of \cite[Prop.3.2(a)]{kn2}), that for a general subspace $V \sub H^0(\R)$ of dimension $r$, the evaluation map $V  \* \O_S \khpil \R$ is generically an isomorphism and drops rank along an irreducible curve $C \in |\det \R|$, which is smooth away from the points where $\R$ is not globally generated, and the cokernel is a torsion free sheaf of rank one.

\begin{defn} 
{\rm A vector bundle $\R$ of rank $r \geq 1$ on a surface $S$ is said to be} good {\rm if it is globally generated off a finite set, $h^0(\R^*)=0$, $c_1(\R)^2 >0$ and there is a subspace $V \sub H^0(\R)$ of dimension $r$ such that the evaluation map $V  \* \O_S \hpil \R$ is injective and drops rank along 
a} smooth, {\rm irreducible curve $C \in |\det \R|$.} 
\end{defn}

For our purposes it will be sufficient to know the following

\begin{lemma} 
\label{lemma:good}
Let $\L$ be a line bundle on a smooth regular surface with $\L^2 >0$ and $Z \subset S$ a zero-dimensional subscheme such that $|\I_Z \* \L| \neq \emptyset$, $\dim (\Bs |\I_Z \* \L|) =0$ and $\Bs |\I_Z \* \L|$ is curvilinear. Let $\R$ be the dual of the kernel of the evaluation map  $H^0(\I_Z \* \L) \* \O_S \khpil \I_Z \* \L$. Then $\R$ is good.
\end{lemma}

\begin{proof}
Let $V = H^0(\I_Z \* \L)$ and set $W= \Bs |\I_Z \* \L|$. Then, by assumption, $V \iso H^0(\I_W \* \L)$ and we have a short exact sequence
\begin{equation*}
0 \hpil \R^* \hpil V \* \O_S \hpil \I_W \* \L \hpil 0,
\end{equation*}
where $\R^*$ is well-known to be locally free and is seen to satisfy $h^0(\R^*)=0$ and $c_1(\R) \sim \L$.

By Bertini's theorem \cite[Prop.1.1]{ei}, the general element in $|\I_W \* \L|$ is a smooth irreducible curve. Pick any such curve $C$ and 
consider the restriction map sequence
\begin{equation}
\label{eq:good1}
0 \hpil \O_S \hpil \I_W \* \L  \hpil \O_C(\L)(-W) \hpil 0.
\end{equation}
Taking evaluation maps in \eqref{eq:good1}, the snake lemma yields a short exact sequence
\[ \xymatrix{
0 \ar[r] & \R^* \ar[r] &  V_C \* \O_S \ar[r]^{ev_{V_C \ \ }} &   \O_C(\L)(-W) \ar[r] & 0} \]
where $V_C := \Im \{ H^0(\I_W \* \L) \khpil H^0(\O_C(\L)(-W))$. Dualizing, we obtain
\[ \xymatrix{
0 \ar[r] & V_C^* \* \O_S \ar[r] & \R \ar[r] &  \O_C(W) \ar[r] & 0} \]
and we see that $V_C^* \sub H^0(\R)$ is the desired subspace.
\end{proof}
The main application of this construction will be the following useful tool
\begin{prop} 
\label{prop:bassoc}
Let $\R$ be a vector bundle of rank at least $2$ on an Enriques surface $S$, with $\R$ globally generated off a finite set, $h^0(\R^*)=0$ and $c_1(\R)^2 >0$. Set $c(\R)=c_2(\R)-2(\rk \R-1)$ and $M = \det \R$. Then $c(\R) \geq 0$.

If $c(\R) \leq 1$, then $h^1(\R^*)=0$ and furthermore:
\begin{itemize}
\item[(i)] If $c(\R)=0$ and $\phi(M) \geq 2$, then $M^2 \leq 8$, and if equality occurs, then
$M \eqv 2M_0$, with $M_0^2=2$. 
\item[(ii)] If $c(\R)=1$, $h^0(\R(K_S)) \geq 2$, $\phi(M) \geq 2$ and $\R$ is good, then one of the following occurs:
\begin{itemize}
\item[(ii-a)] $M^2=12$ and $M \sim 3E_1+2E_2$, with $E_i >0$, $E_i^2=0$ and $E_1.E_2=1$.
\item[(ii-b)] $(M^2, \phi(M))=(10,3)$, $\rk \R=2$ or $4$ and $M \sim 2E +R+K_S$, with $R$ nodal, 
           $E >0$, $E^2=0$ and $E.R=3$.
\item[(ii-c)] $(M^2, \phi(M))=(10,2)$.
\item[(ii-d)] $M^2 \leq 8$.
\end{itemize}
\end{itemize}
\end{prop}
\begin{proof} As mentioned above, for a general subspace $V \sub H^0(\R)$ of dimension $\rk \R$, the evaluation map yields 
a short exact sequence
\begin{equation} 
\label{eq:bassoc1}
0 \hpil V  \* \O_S \hpil \R \hpil B \hpil 0,
\end{equation}
where $B$ is a torsion free sheaf of rank one on some reduced irreducible $C \in |\det \R|$, given by wedging the sections in a basis of $V$. Moreover $B$ is globally generated off a finite set whence $h^0(B)>0$. Dualizing we obtain
\begin{equation} 
\label{eq:bassoc2}
0 \hpil \R ^* \hpil V^* \* \O_S \hpil A \hpil 0,
\end{equation}
where $A$ is a torsion free sheaf of rank one on $C$. Moreover, if $\R$ is good, then we can and will assume that $C$ is smooth and $A$ and $B$ are line bundles with $B= \N_{C/S}-A$.

One easily sees that $c_2(\R) = \deg A$, $h^1(A)=h^0(\R \* \K_S)$ and $\rk \R + h^1(\R^*)=h^0(A)$, so that $c(\R) =  \Cliff A +2h^1(\R^*)  \geq \Cliff A$. 

Now $c(\R) \geq 0$ follows from \cite[Prop.3.2(a)]{kn2}.
If $h^1(A)=0$, then one easily sees, as in the proof of \cite[Prop.3.2(a)]{kn2}, that $c(\R) \geq 2$. Therefore, if $c(\R) \leq 1$, we have that $A$ is special and, using \cite[Thm.A,appendix]{eks}, we deduce that $h^1(\R^*)=0$ and $c(\R) =  \Cliff A$.

We have left to prove (i) and (ii).

We first prove (i). If $c(\R)=0$ it follows again from \cite[Thm.A,appendix]{eks} that either $p_g(C)=0$, or $p_g(C) \geq 1$ and either $A = \O_C$ or $A = \omega_C$ or $C$ possesses a $g^1_2$ (that is, a line bundle $\L$ with $\deg \L =2$ and $h^0(\L)=2$). The case $p_g(C)=0$ cannot happen for general $C$ constructed as above, because, as $\R$ is globally generated off a finite set, it follows that $C$ moves, whence $S$ would be covered by (singular) rational curves, a contradiction. Therefore, the general $C$ obtained as above has $p_g(C) \geq 1$, whence also $p_a(C) \geq 1$. Now $h^0(A) = \rk \R \geq 2$ so it cannot be $A = \O_C$.  If $A = \omega_C = \O_C(C+K_S)$, twisting the exact sequence \eqref{eq:bassoc2} by $\O_S(-C-K_S)$ and dualizing we deduce that $B \otimes \O_C(C+K_S) \cong {\mathcal Ext}^1_{\O_S}(\O_C, \O_S) \cong \O_C(C)$ and therefore $B \cong \O_C(K_S)$. Now it is easily seen that $H^0(\O_C(K_S)) = 0$, a contradiction. Therefore, the general $C$ obtained as above has a $g^1_2$, which is necessarily base-point free, since $C$ is not rational. It is standard that any $Z$ in this linear system poses dependent conditions on $|M+K_S|$, and since $\phi(M) \geq 2$, then $|M+K_S|$ is base-point free, and if $M^2 \geq 8$ we can apply \cite[Prop.3.7]{kn1} and find that there is an effective divisor $D$ on $S$ passing through $Z$ and such that
$2D^2 \leq D.M \leq D^2+2 \leq 4$.
Also, as $C$ is irreducible, we have $D.C=D.M \geq 2$, and we get the only two possibilities $(D^2, D.M)=(0,2)$ or $(2,4)$. Since $\phi(M) \geq 2$, we have $h^0(D)=h^0(D+K_S)=1$ in the first case, whence there are only finitely many such divisors $D$ with $D.M=2$. Choosing $Z$ general, which we can do since the $g^1_2$ is base-point free, we can avoid this case. Hence we are in the second case and $M \eqv 2D$ by the Hodge index theorem. This proves (i).

Now we prove (ii). 
Since $\R$ is good we can assume that $C$ is smooth and $A$ a special line bundle on $C$ with $\Cliff A = 1$. By assumption $h^1(A)=h^0(\R \* \K_S) \geq 2$, whence $\Cliff C \leq 1$. If $\Cliff C=0$, then $C$ is hyperelliptic, and, as in the proof of (i), we get $M^2 \leq 8$.

We can therefore assume $\Cliff C=1$. We can furthermore assume that either $M^2 \geq 12$ or $(M^2, \phi(M))=(10,3)$, since otherwise we would be in one of the cases (ii-c) or (ii-d).

We first show that the case $(M^2, \phi(M), \rk \R)=(10,3,3)$ does not occur. 

Indeed, in this case $|A|$ is a $g^2_5$, that is base-point free since $\Cliff C=1$, whence $A$ is very ample since $g(C) = 6$. Hence $C$ is isomorphic to a smooth plane quintic. By \cite[Lemma2.18]{kl1} we can choose an $E >0$ such that $E^2=0$, $E.M=3$, $|M-E|$ is base-point free and $h^0(M-E)=3$. As in the proof of \cite[Prop.4.13]{kl1} we have that there is an effective divisor $Z_3 \subset C$ of degree $3$ such that $|\O_C(M-E)(-Z_3)|$ is a base-point free, complete $g^1_4$ on $C$, call it $A_0$. It is well-known that any $g^1_4$ on a smooth plane quintic comes from projecting from a point on the curve, that is there is a point $y \in C$ such that
$A \sim A_0 + y \sim \O_C(M-E)(-Z_3+y)$.
If $y \in \Supp Z_3$, then $\O_C(M-E) > A$, which is impossible, since $h^0(\O_C(M-E))=h^0(A)=3$ and $\O_C(M-E)$ is base-point free. Therefore $y \not \in \Supp Z_3$, which will be useful later on. 

Let now $\E =\E(C,A_0)$, which is locally free of rank two and sits into
\begin{equation} 
\label{eq:bassoc3}
0 \hpil H^0(A_0)^{\ast} \* \O_S \hpil \E \hpil \N_{C/S}-A_0 \hpil 0.
\end{equation}
Now  $h^1(-E)=0$ (as $E.M=\phi(M)=3$ implies $h^0(E)=h^0(E+K_S)=1$) whence $h^0(\E(-E)) = h^0((\N_{C/S}-A_0)(-E))=h^0(\O_C(M-E)(-A_0)) = h^0(\O_C(Z_3)) = 1$ and $\E$ is globally generated off a finite set. We also claim that
\begin{equation} 
\label{eq:bassoc6}
h^0(\E)=3.
\end{equation}
To prove this, we saturate the inclusion $0 \khpil \O_S(E) \khpil \E$ to obtain a short exact sequence
\begin{equation} 
\label{eq:bassoc4}
0 \hpil \O_S(\Delta)  \hpil \E \hpil \I_X (M-\Delta) \hpil 0 
\end{equation}
with $\Delta \geq E$ and $X \subset S$ a $0$-dimensional subscheme. Furthermore $c_2(\E)=4=\Delta.(M-\Delta) + \deg X \geq
\Delta.(M-\Delta)$. Moreover $|M-\Delta|$ is base-component free. If $M-\Delta \sim 0$ then \eqref{eq:bassoc4} yields the contradiction $3 = h^0(M-E) \leq  h^0(\E(-E)) = 1$. Therefore $(M-\Delta)^2 \geq 0$ and $h^0(M-\Delta) \geq 2$. If $(M-\Delta)^2 = 0$ we get the contradiction $4 \geq (M-\Delta).M \geq 2 \phi(M) = 6$. Hence $(M-\Delta)^2 > 0$ and $\Delta.M \geq E.M=3$, so that $\Delta^2 \geq -4 + \Delta.M \geq -1$, giving $\Delta^2 \geq 0$. Now one easily sees, by the Hodge index theorem and $M^2=10$, that $\Delta^2 \geq 2$ implies $\Delta^2=(M-\Delta)^2=2$ and $\Delta.(M-\Delta)=3$. But then $\Delta.M=5$, yielding $\phi(M)=2$, a contradiction. Therefore, we only have the two possibilities
\begin{equation} 
\label{eq:bassoc5}
(\Delta^2, \Delta.(M-\Delta), (M-\Delta)^2, \deg X)=(0,3,4,1), \; (0,4,2,0). 
\end{equation}
In both these cases we see that $h^0(\Delta)=1$ since $\Delta.M < 2\phi(M)=6$, whence $h^1(\Delta)=0$. Moreover, in the first case, as $(M-\Delta).M =7 < 3\phi(M)=9$, we must have $\phi(M-\Delta)=2$. Therefore $|M-\Delta|$ is base-point free in this case. It follows that in both 
cases of \eqref{eq:bassoc5} we have $h^0(\I_X (M-\Delta))=2$. Therefore \eqref{eq:bassoc6} follows by taking cohomology in  \eqref{eq:bassoc4}.

Comparing with \eqref{eq:bassoc3} we see that $h^0(\N_{C/S}-A_0)=1$. Taking evaluation maps in \eqref{eq:bassoc3} we get that the scheme where $\E$ fails to be globally generated is precisely the unique member $T_6 \in |\N_{C/S}-A_0|$ of length $6$. Moreover we note that
\begin{equation} 
\label{eq:bassoc6'}
T_6 = (E \cap C) \cup Z_3.
\end{equation}
The inclusion $A_0 = A-y \subset A$ yields the exact sequence
\begin{equation} 
\label{eq:bassoc7}
 \xymatrix{
0 \ar[r] & \O_S \ar[r] &  \R \ar[r] & \E  \ar[r] & \tau_y \ar[r] & 0 }
\end{equation}
where $\tau_y$ is a torsion sheaf of length one supported only at $y$.

By \eqref{eq:bassoc6} and \eqref{eq:bassoc7} we get $h^0(\R) \leq 4$. On the other hand, using \eqref{eq:bassoc1} and the fact that $h^0(B)>0$, we get $h^0(\R) = 3 + h^0(B) \geq 4$, whence $h^0(\R)=4$ and $h^0(B) = h^0(\N_{C/S}-A)=1$. Furthermore it follows that the scheme where $\R$ fails to be globally generated is precisely the unique member $T_5 \in |\N_{C/S}-A|$ of length $5$. Note that $T_5=T_6-y$ (as divisors on $C$). Since we have seen that $y \not\in \Supp Z_3$ we have $y \in E \cap C$ by \eqref{eq:bassoc6'}, whence
$T_5 = X_2 \cup Z_3, \; \mbox{where} \; X_2 = E \cap C - y$.
Now from
\[ \xymatrix{ 0 \ar[r] & \O_S \ar[r] &  \I_{T_5 / S} \* M \ar[r] & \I_{T_5 / C} \* M \ar[r] & 0 } \]
we see that $\Bs |\I_{T_5} \* M|=T_5$ since $\I_{T_5 / C} \* M \cong A$. We now claim that any $M' \in |M|$ passing through $X_2 \sub E \cap C$ must pass through the whole of $E \cap C$. Assuming the claim for a moment, we get the contradiction $\Bs |\I_{T_5} \* M|=T_6$. To see the claim first note that the exact sequence
\[ 0 \hpil \O_S \hpil \I_{(E \cap C) / S}(M) \hpil \O_C(M-E) \hpil 0 \]
shows that $h^0(\I_{(E \cap C) / S}(M)) = 4$, since $h^0(\O_C(M-E)) = 3$, which is easily seen to follow from $h^0(M-E)=3$. To see the claim it is therefore enough to show that $h^0(\I_{X_2 / S}(M)) = 4$. To this end first note that from the exact sequence
\[ 0 \hpil \O_S \hpil \I_{X_2 / S}(M) \hpil \O_C(M)(-X_2) \hpil 0 \]
we see that $W := \Bs|\I_{X_2 / S} \* M| \sub E \cap C$, since $\O_C(M)(-X_2) \cong \O_C(M-E)(y)$ and $\O_C(M-E)$ is base-point free. Hence $W$ is curvilinear of length at most $3$. Blowing-up  $S$ at most three times, we resolve the base-scheme of $|\I_{X_2 / S} \* M|$ and therefore the resulting linear system is base-point free and not composite with a pencil, so that its general divisor is irreducible by Bertini's theorem. It follows that the general divisor $M_0 \in |\I_{X_2 / S} \* M|$ is irreducible.  Now $h^0(M) = 6$ and $\deg X_2 = 2$, whence, to prove that $h^0(\I_{X_2 / S}(M)) = 4$, we can just show that $|M|$ separates $X_2$. By Reider's theorem \cite[Thm.1]{re}, if $|M|$ does not separate $X_2$, there exists an effective divisor $G$ on $S$ such that $X_2 \subset G$ and either $G^2 = 0, G.M = 1, 2$ or $G^2 = -2, G.M = 0$. But the first case is excluded since $\phi(M) = 3$ and the second since $G.M = G.M_0 \geq 2$.

We have therefore proved that the case $(M^2, \phi(M), \rk \R)=(10,3,3)$ does not occur. 

Now assume $(M^2, \phi(M))=(10,3)$. Then Riemann-Roch yields
\[ h^0(A)-h^1(A)= \rk \R - h^0(\R(K_S))= \deg A-5 = 2(\rk \R-1)-4 = 2\rk \R-6, \]
whence $\rk \R + h^0(\R(K_S))=6$, which yields the two possibilities $(\rk \R, h^0(\R(K_S)))=(h^0(A),h^1(A))=(2,4)$ or $(4,2)$, by assumption. In the first case $|A|$ is a $g^1_3$ and in the second $|\omega_C-A|$ is. Therefore $C$ is trigonal. This also holds for $M^2 \geq 12$, since $\Cliff C = 1$, therefore, in the remaining cases to treat, we have that $C$ is trigonal. Denote a $g^1_3$ by $|A_0|$.

We now use \cite[Prop.3.1]{kl1}. 
If we are in case (a) of that proposition there is a base-component free line bundle $D$ such that 
$2D^2 \leq D.M \leq D^2+3 \leq 6 \; \mbox{and} \; \O_C(D) \geq A_0$.
If $D^2=0$, then we must have $\O_C(D) \sim A_0$ and $D.M=3$, but this is impossible, since $D$ has to be a multiple of an elliptic pencil. Therefore $D^2 =2$ and $D.M \leq 5$, whence $\phi(M) \leq 2$. Therefore $M^2 \geq 12$ by our assumptions, and the Hodge index theorem yields $M^2=12$ and $D.M=5$. We can write $D \sim E_1+E_2$ for $E_i >0$, $E_i^2=0$, $E_1.E_2=1$, $E_1.M=2$ and $E_2.M=3$. Now $(M-3E_1)^2=0$, whence $M \sim 3E_1+F$ for some $F>0$ with $F^2=0$ by \cite[Lemma2.4]{kl1}. But $3=E_2.M=3+E_2.F$ yields $E_2.F=0$, whence $F \eqv 2E_2$ by \cite[Lemma2.1]{klvan}. Therefore $M \eqv 3E_1+2E_2$ and we are in case (ii-a).

If we are in case (b) of \cite[Prop.3.1]{kl1} then $M^2=10$ and we have $M \sim 2N+R+K_S$ for some nodal $R$ and $N>0$ such that $N.M \geq \phi(M)=3$, $N.(N+R)=3$ and $|N+R+K_S|$ is base-component free. Since $3 \leq N.M = N^2 +N.(N+R)$ we must have $N^2 \geq 0$, and similarly, as $h^0(N+R+K_S) \geq 2$, we must have $6=2\phi(M) \leq (N+R).M$, whence $(N+R)^2 \geq 4$. Using $M^2 = 10$ we see that the only possibility is $N^2=0$ and $(N+R)^2=4$, therefore $N.R=3$, which shows that we are in case (ii-b).
\end{proof}

\section{Plane curves}
\label{sec:plane}

It it easily seen that an Enriques surface can contain plane curves of degree up to $4$ and it is a nontrivial result of Stagnaro \cite{sta} and Umezu \cite{umeq} that the same happens for plane quintics (however never general in their linear system \cite[Prop.4.13]{kl1}). On the other hand, in this section we will prove that, on an Enriques surface, there are no smooth plane curves of degree at least $6$. 
We remark that in \cite{glm1} the same result was proved for degree at least $9$. The present proof is independent of that one.

\begin{prop}
\label{plane}
On an Enriques surface there are no curves which are isomorphic to a smooth plane curve of degree $d \geq 6$.
\end{prop}

\noindent {\it Proof.} Assume, to get a contradiction, that $C$ is a smooth plane curve of degree $d \geq 6$ lying on an Enriques surface $S$, that is $C$ has a very ample line bundle $A$ with $h^0(A) = 3$, $\deg A = d$. We set $L = \O_S(C)$ and let $E > 0$ be a nef divisor such that $E.L
= \phi(L)$.

If $d = 6$, then $\gon(C) = 5$ and $L^2 = 18$, whence $3 \leq \phi(L) \leq \lfloor \sqrt{18} \rfloor = 4$.

We will now rule out the case $d = 6$, $\phi(L) = 3$.
Set $M = L - 2E$, so that $M^2 = 6, M.L = 12$, hence $H^2(M) = H^0(- M + K_S) = 0$. From the exact sequence
\[ 0 \hpil - M \hpil 2E \hpil \O_C(2E) \hpil 0 \]
we see that if $H^1(- M) = 0$ we are done: In fact this implies that $|\O_C(2E)|$ is a base-point free complete $g^1_6$ on $C$, but this is not possible on a smooth plane sextic, as any such $g^1_6$ is contained in the linear series cut out by the lines (this is a well-known fact, see for example
\cite{lp}). Suppose then that $H^1(- M) \neq 0$. By \cite[Cor.2.5]{klvan} $M$ is not quasi-nef, that is there is a $\Delta > 0$ such that $\Delta^2 = -2$ and $\Delta.M \leq -2$, whence, by \cite[Lemma2.3]{kl1}, setting $k = - \Delta.M \geq 2$, there exists an $A > 0$ such that $A^2 = 6$, $A.\Delta = k$ and $M \sim A + k \Delta$. Now $0 \leq L.\Delta = 2E.\Delta + M.\Delta$ whence $E.\Delta \geq 1$. From $3 = E.M = E.A + k E.\Delta$ we see that the only possibility is $k = 2$ and $E.\Delta = 1$, therefore $L.\Delta = 2E.\Delta + M.\Delta = 0$ and the Hodge index theorem implies that $L \eqv 6E + 3\Delta$. In particular $2E + \Delta$ is nef, $(2E + \Delta)^2 = 2$, hence $h^0(2E + \Delta) = 2$, $h^1(2E + \Delta) = 0$. Also $L - 2E - \Delta \eqv 2(2E + \Delta)$ whence $h^i(2E + \Delta - L) = 0$, $i = 0, 1$. From the exact sequence
\[ 0 \hpil 2E + \Delta - L \hpil 2E + \Delta \hpil \O_C(2E + \Delta) \hpil 0 \]
we see that we are done because, as above, $|\O_C(2E + \Delta)|$ is a base-point free complete $g^1_6$ on $C$, leading to the same contradiction.

We now proceed with the proof of the other cases.
Since $(E+ K_S)_{|C} > 0$ we have
\begin{eqnarray*}
h^0(\N_{C/S} - A) & = & h^0(\omega_C(- A) (K_S) \geq h^0(\omega_C(- A) (- (E+ K_S)_{|C})) \geq \\
\nonumber & \geq & h^0(\omega_C - A) - \phi(L) = h^1(A) - \phi(L) = g - d + 2 - \phi(L) = \\
& = & \frac{1}{2}(d - 1)(d - 2) - d + 2 - \phi(L) \geq \frac{1}{2}d(d - 5) + 3 - \lfloor \sqrt{d(d - 3)} \rfloor,
\end{eqnarray*}
whence we have shown that
\begin{equation}
\label{eq:d>6}
\mbox{if} \; d \geq 7 \; \mbox{then} \; h^0(\N_{C/S} - A) \geq 5
\end{equation}
and
\begin{equation}
\label{eq:pl2}
\mbox{if} \; d = 6 \; \mbox{then} \; (L^2, \phi(L)) = (18,4) \; \mbox{and} \; h^0(\N_{C/S} - A) \geq 2.
\end{equation}
Set $\E = \E(C,A)$. By \eqref{eq:eca}, (\ref{eq:d>6}) and (\ref{eq:pl2}) we get \[ h^0(\E) \geq 8 \; \mbox{if} \; d \geq 7, \; \mbox{and} \; h^0(\E) \geq 5 \;
\mbox{if} \; d =6. \]
Moreover, as $h^0(\N_{C/S} - A) > 0$, we get by \eqref{eq:eca} that $\E$ is globally generated off a finite set.

We first need the following

\begin{lemma}
\label{useful4}
Let $s \in H^0(\E)$ be a nonzero section. Denote by $D \geq 0$ be the divisorial subscheme of the zero locus of $s$. Then we have an exact sequence as in (\ref{eq:E8}) with $h^0(D) = 1$, $D \leq L$ and $c_1(\F)^2 - 4c_2(\F) = d(d - 7) + 2D.L - 3D^2 + 4\length(\tau)$.
\end{lemma}

\begin{proof}
By Lemma \ref{useful3} we get an exact sequence as in (\ref{eq:E8}) with $\rk \F = 2$, $M = \det \F$ nontrivial, base-component free and $L \sim M + D$. The formula for $c_1(\F)^2 - 4c_2(\F)$ is immediate from \eqref{eq:E11}.
Assume, to get a contradiction, that $h^0(D) \geq 2$. Then Lemma \ref{bigger0} and (\ref{eq:E11}) yield $D^2 > 0$, $M^2 > 0$ and $d - 4 = c \geq D.M - 2 + \length(\tau)$, hence $D.M \leq d -2$. From the Hodge index theorem $D^2M^2 \leq (D.M)^2 \leq (d - 2)^2$, hence, as $D^2 \geq 2$, $M^2 \geq 2$ we have $D^2 + M^2 \leq \frac{1}{2}(d - 2)^2 + 2$ with equality if and only if either $D^2 = 2$ or $M^2 = 2$. Also
\begin{equation}
\label{eq:E16}
d(d - 3) = L^2 = D^2 + M^2 + 2D.M \leq \frac{1}{2}(d - 2)^2 + 2 + 2(d - 2) = \frac{1}{2}d^2,
\end{equation}
whence $d = 6$ and we must have equalities all along, in particular $D.M = d - 2 = 4$ and $(D^2, M^2) = (2, 8)$ or $(8, 2)$. By the Hodge index theorem we deduce that either $L \eqv 3D$ with $D^2 = 2$ or $L \eqv 3M$ with $M^2 = 2$, but this is a contradiction since, by (\ref{eq:pl2}), $\phi(L) = 4$.
\end{proof}

\noindent {\it Continuation of the proof of Proposition} \ref{plane}.
Now consider the set
\begin{equation} 
\label{eq:Q}
Q = \{\Gamma : \Gamma > 0, \Gamma.L \leq L^2 \ \mbox{and either} \ \Gamma \ \mbox{is nodal} \
\mbox{or} \ |2\Gamma| \ \mbox{is a genus one pencil} \}.
\end{equation}

We note that $Q$ is a finite set Q by standard arguments. We define
\begin{equation} 
\label{eq:pi}
\Pi = \bigcup\limits_{\Gamma \in Q} \Gamma.
\end{equation}
Then we have

\begin{lemma}
\label{xylemma}
If $h^0(\E) \geq 6$ then there are two distinct points $x$ and $y$ on $S$ lying outside of $\Pi$ and a section $s$ of $\E$ vanishing at $x$ and $y$.
\end{lemma}

\begin{proof}
This follows,  almost verbatim, from the proof of \cite[(2.10)]{gl1}.
\end{proof}

\renewcommand{\proofname}{Conclusion of the proof of Proposition \ref{plane}}
\begin{proof}
Consider the two cases
\begin{itemize}
\item[(I)] $h^0(\E) \geq 6$ (in particular if $d \geq 7$),
\item[(II)] $d = 6$ and $h^0(\E) = 5$
\end{itemize}
and choose a nonzero section $s$ of $H^0(\E)$ subject to the following conditions:
\begin{itemize}
\item[(I)] $s \in H^0(\E)$ vanishes on $x$ and $y$ as in Lemma \ref{xylemma},
\item[(II)] $s \in H^0(\E)$ vanishes on $E$, where $E.L = \phi(L) = 4$.
\end{itemize}
The existence of such a section follows by Lemma \ref{xylemma} in case (I) above, while, in case (II), it follows from  \eqref{eq:eca} twisted by $\O_S(- E)$ and by $h^0((\N_{C/S} - A) - E_{|C}) = h^0(\omega_C(- A) - (E+K_S)_{|C}) \geq h^1(A) - \phi(L) \geq 2$.
Now by Lemma \ref{useful4} we have an exact sequence as in (\ref{eq:E8}) with $h^0(D) = 1$, in particular $D^2 \leq 0$ by Riemann-Roch, and $c_1(\F)^2 - 4c_2(\F) = d(d - 7) + 2D.L - 3D^2 + 4 \length(\tau) \geq d(d - 7) + 2D.L + 4 \length(\tau)$.
Let $\Gamma$ be an irreducible component of $D$. Then either $\Gamma$ is nodal or $\Gamma^2 \geq 0$. In the latter case $h^0(\Gamma) \leq h^0(D) = 1$ therefore $h^0(\Gamma) = 1$ and $\Gamma^2 = 0$ by Riemann-Roch. Then $\Gamma$ is indecomposable of canonical type and by
\cite[Prop.3.1.2]{cd} we get that $|2\Gamma|$ is a genus one pencil. Since $D - \Gamma \geq 0$ and $L - D > 0$ we get, by the nefness of $L$, that $L.\Gamma \leq L.D \leq L^2$, therefore $\Gamma \in Q$.

In case (I), by the choice of $x$ and $y$, we have that $x \not\in D$, $y \not\in D$ hence $x, y \in \Supp(\tau)$. Therefore $\length(\tau) \geq 2$ and $c_1(\F)^2 - 4c_2(\F) \geq d(d - 7) + 2D.L + 8 \geq 2$. In case (II) $D \geq E$, whence $D.L \geq E.L = 4$ and again $c_1(\F)^2 - 4c_2(\F) \geq 2$.

As in the proof of \cite[Prop.3.1]{kl1} there are two line bundles $M_1$ and $M_2$ and a zero-dimensional subscheme $W \subset S$ fitting in an exact sequence
\[ 0 \hpil M_1 \hpil \F \hpil \I_W \* M_2 \hpil 0, \]
with $M = \det \F \sim M_1 + M_2$ and such that either $M_1 \geq M_2$ or $W = \emptyset$ and the sequence splits. Moreover we find $(M_1 - M_2)^2 = c_1(\F)^2 - 4c_2(\F) + 4 \length(W) \geq 2$. Hence, in any case, without loss of generality, we can assume, by Riemann-Roch, that $M_1 > M_2$.
Recall that $M_2$ is base-component free and nontrivial, whence $M_2$ is nef with $h^0(M_2) \geq 2$.

{\bf Case (I)}. We have $c_2(\F) = M_1.M_2 + \length(W)$, whence by (\ref{eq:E11}),
\begin{equation}
\label{eq:E18}
c =  D.M + \length(\tau) + c_2(\F) - 4 \geq D.M_1 + D.M_2 + M_1.M_2 - 2.
\end{equation}
Assume first that $M_2^2 = 0$. Then $M_2 \sim mP$ for an elliptic pencil $|P|$ and an integer $m \geq 1$ and $M_2.(D + M_1) = M_2.L = mP.L \geq 2m\phi(L)$. Then (\ref{eq:E18}) yields $c \geq D.M_1 + 2m\phi(L) - 2 \geq D.M_1 + 2\phi(L) - 2$, whence $D.M_1 \leq -1$ by (\ref{eq:E2b}) and
it follows that $D > 0$. Since $D.L = D.M_1 + D.M_2 + D^2 \leq -1 + D.M_2 + D^2$ and $L$ is nef we must have $D.M_2 + D^2 \geq 1$. Recall that $D^2 \leq 0$. Now if $D.M_2 \leq 2$, then $D^2 = 0$ and $2 \leq \phi(L) \leq D.L \leq 1$ a contradiction. Hence $D.M_2 \geq 3$ and $(D + M_2 + K_S)^2
\geq 4$, $h^0(D + M_2 + K_S) \geq 3$ by Riemann-Roch and \eqref{eq:E18} gives
$c \geq D.M_2 + (D + M_2).M_1 - 2 \geq (D + M_2).M_1 + 1$, contradicting Lemma \ref{cliffhanger}.

Then $M_2 ^2 > 0$, whence by Riemann-Roch, $h^0(M_2) \geq 2$, $h^0(M_2 + K_S) \geq 2$ and also $h^0(M_1) \geq h^0(M_2) \geq 2$, $h^0(M_1 + K_S) \geq h^0(M_2 + K_S) \geq 2$. By (\ref{eq:E18}) we have $c \geq D.M_1 - 2 + (D + M_1).M_2$, hence by Lemma \ref{cliffhanger} we have $D.M_1 \leq 2$ and similarly $D.M_2 \leq 2$.

If $D.M_1 \leq -1$ then $D > 0$. Therefore $D.L = D.M_1 + D.M_2 + D^2 \leq 1 + D^2$, a contradiction both if $D^2 = 0$ and if $D^2 \leq -2$. Hence $D.M_1 \geq 0$ and similarly $D.M_2 \geq 0$. Then $d - 4 = c \geq (D + M_1).M_2 - 2$ hence $(D + M_1).M_2 \leq d - 2$ and similarly $(D + M_2).M_1 \leq d - 2$. From Lemma \ref{cliffhanger} we deduce that $h^1(D + M_2) \leq 1$, whence $(D + M_2)^2 \geq 0$ by Riemann-Roch and, if equality occurs, we get the contradiction $2 \phi(L) \leq (D + M_2).L = (D + M_2).M_1 \leq d - 2 = \gon(C) - 1$. This shows that $(D + M_2)^2 > 0$ and similarly $(D + M_1)^2 > 0$.

From the Hodge index theorem we get $(D + M_1)^2M_2^2 \leq ((D + M_1).M_2)^2 \leq (d - 2)^2$, so that $(D + M_1)^2 + M_2^2 \leq \frac{1}{2}(d - 2)^2 + 2$ and we get the same contradiction as in (\ref{eq:E16}) above and the following lines.
This concludes the proof of Proposition \ref{plane} in Case (I).

{\bf Case (II)}. First write $M_1 \sim M_2 + M_3$ with $M_3 > 0$ and $M_3^2 > 0$ and set $l = \length(\tau) + \length(W)$. We have $c_2(\F) = M_1.M_2 + \length(W)$, whence by (\ref{eq:E11}) we get
\begin{eqnarray}
\label{eq:E19}
\nonumber 2 = \ c & = & D.M + \length(\tau) + c_2(\F) - 4 = \\
& = & D.M_1 + D.M_2 + M_1.M_2 - 4 + \length(\tau) + \length(W) = \\
\nonumber & = & 2D.M_2 + D.M_3 + M_2^2 + M_2.M_3 - 4 + l.
\end{eqnarray}
Moreover, by our assumptions, (\ref{eq:E8}) and $h^2(D) = 0$ we have
\begin{eqnarray*}
5 = h^0(\E) & \geq & \chi(D) + h^0(\F) - \length(\tau) \geq \\
& \geq & \chi(D) + \chi(M_1) + h^0(M_2) - \length(W)  - \length(\tau) = \\
& = & 3 + \frac{1}{2}D^2 + M_2^2 + \frac{1}{2}M_3^2 + M_2.M_3 + h^1(M_2) - l
\end{eqnarray*}
whence, combining with (\ref{eq:E19}),
\begin{eqnarray}
\label{eq:bis}
\nonumber 4 = \phi(L) = E.L & \leq & D.L = D^2 + 2D.M_2 + D.M_3 \leq \\
& \leq & 8 + \frac{1}{2}D^2 - 2M_2^2 - \frac{1}{2}M_3^2 - 2M_2.M_3 - h^1(M_2).
\end{eqnarray}

Assume first $M_2^2 = 0$. Then $M_2 \sim mP$ for an elliptic pencil $|P|$ and an integer $m \geq 1$ and we have $h^1(M_2) = m$, hence, since $M_3^2 \geq 2$ and $P.M_3 \geq 2\phi(M_3) \geq 2$, (\ref{eq:bis}) yields the contradiction
$4 \leq 7 + \frac{1}{2}D^2 - 2m P.M_3 - m \leq 7 - 5m \leq 2$.

This shows that $M_2^2 > 0$, hence $h^1(M_2) = 0$, $M_2.M_3 \geq 2$ and (\ref{eq:bis}) yields the contradiction
$4 \leq - \frac{1}{2}M_3^2 \leq -1$.
This concludes the proof of case (II) and of Proposition \ref{plane}.
\end{proof}
\renewcommand{\proofname}{Proof}

\begin{rem}
\rm{The fact that there are no smooth plane sextics on an Enriques surface answers positively the question raised in \cite[Rem.3.9]{glm1}. Also, for the same reason, in the latter article, Lemma 3.1 is no longer needed.}
\end{rem}

To end the section we will show a result that will have applications in the study of Gaussian maps of these curves (\cite[Proof of Prop.5.14]{klGM}).

\begin{prop}
\label{g^2_6particolari}
Let $L$ be a base-point free line bundle on an Enriques surface with $(L^2, \phi(L)) = (14, 3)$ or $(16, 2)$. Then the general curve in $|L|$ possesses no base-point free complete $g^2_6$.
\end{prop}

\begin{proof}
Assume to get a contradiction that a general smooth irreducible curve $C \in |L|$ has a base-point free line bundle $A$ with $h^0(A) = 3$ and $\deg A = 6$. Set $\E = \E(C,A)$ as usual.

{\bf Case 1: $(L^2, \phi(L)) = (14, 3)$.}

Take a nef $E > 0$ with $E.L = 3$ and $E^2 = 0$. Then from  \eqref{eq:eca}, using $h^1(- E) = 0$, we get
\begin{equation}
\label{eq:8.5}
h^0(\E(- E)) = h^0(\omega_C - A - (E+K_S)_{|C}) \geq h^1(A) - 3 = 1,
\end{equation}
so that there is a nonzero section in $H^0(\E)$ which vanishes along $E$. By Lemma \ref{useful3} we get a sequence as in (\ref{eq:E8}) with $D \geq E$ and a decomposition $L \sim D + M$ where $M = \det \F$. Using the same notation as in that lemma, we have, for $l = \length(\tau)$,
\begin{equation}
\label{eq:8.10}
6 = c_2(\E) = D.M + c_2(\F) + l.
\end{equation}
If $M^2 = 0$, then $M \sim mP$ for an elliptic pencil $P$ and an integer $m \geq 1$. We have $P.D = P.L \geq 2\phi(L) = 6$, whence from (\ref{eq:8.10}) and \cite[Prop.3.2]{kn2} we get the contradiction
\[ 6 = D.M + c_2(\F) + l \geq mP.D - 2m + 4 \geq 4m + 4 \geq 8. \]
Hence $M^2 >0$ and by \cite[Prop.3.2]{kn2} we have $c_2(\F) \geq 2$.
Combining with (\ref{eq:8.10}) we get
\begin{equation}
\label{eq:8.15}
D.M \leq 4 - l \leq 4.
\end{equation}
We divide the rest of the proof of Case 1 into the two subcases:
\begin{itemize}
\item[(1A)] $h^0(D) \geq 2$ or $h^0(D + K_S) \geq 2$,
\item[(1B)] $h^0(D) = h^0(D + K_S) =1 $.
\end{itemize}

{\bf Case (1A):} We have $D.L \geq 2\phi(L) = 6$, whence by (\ref{eq:8.15}) we have
\begin{equation}
\label{eq:8.17}
D^2 = D.(L - M) \geq 2.
\end{equation}

We now claim that $h^1(- D) = 0$. Indeed, if not, by \cite[Cor.2.5]{klvan}, there exists a $\Delta > 0$ with $\Delta^2 = -2$ and $\Delta.D \leq -2$. Then
$\Delta.M = \Delta.(L - D) \geq 2$, since $L$ is nef. It follows that $(M + \Delta)^2 \geq 4$ and, by \cite[Lemma2.1]{klvan} and (\ref{eq:8.15}), $0 \leq (M + \Delta).D \leq 2$, but this contradicts the Hodge index theorem.
Therefore, from
\begin{equation}
\label{eq:8se}
0 \hpil - D \hpil M \hpil \O_C(M) \hpil 0
\end{equation}
and Lemma \ref{useful3} we find
\begin{equation}
\label{eq:8.21}
h^0(M) = h^0(\O_C(M)) \geq h^0(A) = 3,
\end{equation}
hence, in particular
\begin{equation}
\label{eq:8.21'}
M^2 \geq 4.
\end{equation}
From (\ref{eq:8.15}) and (\ref{eq:8.17}) and the fact that $L^2 = D^2 + M^2 + 2D.M \geq 6 + 2D.M$ we see that if $D.M = 4$ then $D^2 = 2$ and $M^2 = 4$. At the same time we see from the Hodge index theorem that $D.M \leq 3$ implies $D^2 = 2$ and $M^2 = 4$, which yields $L^2 \leq 12$, a contradiction. Hence $D.M = 4$, $D^2 = 2$ and $M^2 = 4$. In particular $M.L = 8$. It follows that $\phi(M) = 2$, because if $\phi(M) = 1$, by \cite[Lemma2.4]{kl1}, we can write $M \sim 2M_1 + M_2$ with $M_i > 0$, $M_i^2 = 0$ and $M_1.M_2 = 1$. But then $M.L \geq 3\phi(L) = 9$, a contradiction. Therefore $\phi(M) = 2$, so that $M$ is base-point free by \cite[Thm.4.4.1]{cd} and $\O_C(M)$ is base-point free as well. But by (\ref{eq:8.21}) we have $h^0(\O_C(M)) = h^0(A)$, and since $\O_C(M) \geq A$ by Lemma \ref{useful3}, we must have $\O_C(M) = A$, whence $M.L = \deg A = 6$, a contradiction.

{\bf Case (1B):} We have $D^2 \leq 0$ and $D.L \geq 3$ since $D \geq E$. At the same time (\ref{eq:8.15}) gives $D^2 = D.L - D.M \geq -1$, whence $D^2 = h^1(D) = h^1(D + K_S) = 0$,
the latter vanishings from Riemann-Roch. From (\ref{eq:8se}) we find that (\ref{eq:8.21}) and (\ref{eq:8.21'}) are still valid. Moreover
\begin{equation}
\label{eq:8.25}
D.M = D.L \geq 3.
\end{equation}

By (\ref{eq:8.10}) we get
\begin{eqnarray}
\label{eq:8.36}
\chi(\F \* \F^*) & = & c_1(\F)^2 - 4c_2(\F) + 4 = M^2 - 4(6 - D.M - l) + 4 = \\
\nonumber & = & (L - D)^2 + 4D.M + 4l - 20 = 2D.L + 4l - 6 \geq 0
\end{eqnarray}
where we have used (\ref{eq:8.25}). Hence we have the following two possibilities:
\begin{itemize}
\item[($\alpha$)] $\chi(\F \* \F^*) = 0$,
\item[($\beta$)] $\chi(\F \* \F^*) \geq 2$
\end{itemize}
We will treat these two cases separately.

{\bf Case (1B)($\alpha$):} By (\ref{eq:8.36}) and (\ref{eq:8.25}) we
must have
\begin{equation}
\label{eq:8.26}
D.M = D.L = 3 \; \mbox{and} \; l = 0,
\end{equation}
and it follows that
\begin{equation}
\label{eq:8.26'}
M^2 = L^2 - 2D.L =8  \; \mbox{and} \; \phi(M) = 2.
\end{equation}
Indeed, the latter follows as above, because if $\phi(M)=1$, by \cite[Lemma2.4]{kl1} we can write $M \sim 4M_1 + M_2$ with $M_i > 0$, $M_i^2 = 0$ and $M_1.M_2 = 1$, but then $M.L \geq 5 \phi(L) = 15 > L^2$, a contradiction.
From (\ref{eq:8.10}) we find
\begin{equation}
\label{eq:8.27}
c_2(\F) = 3.
\end{equation}
As in (\ref{eq:8.5}) we find $h^0(\E(-D+K_S)) \geq 1$ and from (\ref{eq:E8}) with $\tau = \emptyset$ we find $h^0(\F(- D + K_S)) = h^0(\E(- D + K_S))$, that is there is $s' \in H^0(\F)$ vanishing on $D + K_S > 0$. Saturating we get
\[ 0 \hpil \O_S(D_1) \hpil \F \hpil \I_X \* M_1 \hpil 0, \]
for an effective divisor $D_1 \geq D + K_S > 0$, $M_1$ a line bundle on $S$ and $X$ a $0$-dimensional subscheme of $S$. Moreover, as $M_1$ is a quotient of $\E$ off a finite set and $h^2(M_1 + K_S) \leq h^2(\F \* \omega_S) \leq h^2(\E \* \omega_S) = 0$ we have
\begin{eqnarray}
\label{eq:8.32}
M_1 \; \mbox{is effective, nontrivial, base-component free and } \\
\nonumber M = D_1 + M_1, \ \Bs |M_1| \sub X \cup \Bs |\N_{C/S} - A| \sub X \cup C.
\end{eqnarray}
Moreover, from (\ref{eq:8.27}) we get
\begin{equation}
\label{eq:8.31}
3 = c_2(\F) = D_1.M_1 + \length(X) \geq D_1.M_1.
\end{equation}
If $M_1^2 = 0$, then $M_1.D_1 = M_1.M \geq 2\phi(M) = 4$ from (\ref{eq:8.26'}), contradicting (\ref{eq:8.31}). Hence, by the nefness of $M_1$, we have
$M_1^2 \geq 2, \; h^1(M_1) = h^1(M_1 + K_S) = 0$.
Since $D_1 \geq D + K_S$ and $M$ is nef we have, using (\ref{eq:8.31}),
$3 = D.M \leq D_1.M = D_1^2 + D_1.M_1 \leq D_1^2 + 3$,
whence $D_1^2 \geq 0$.

If $\phi(M_1) = 1$, then $|M_1|$ has two base points and since $C$ is general in its linear system, it cannot contain any of these two points, whence $\length(X) \geq 2$ by (\ref{eq:8.32}). Then $D_1.M_1 \leq 1$ by (\ref{eq:8.31}), whence $D_1.M_1 = 1$ and $D_1^2 = 0$ by the Hodge index
theorem. But then $D_1.M = D_1.M_1 = 1$, contradicting (\ref{eq:8.26'}).
Hence $\phi(M_1) \geq 2$, so that $M_1^2 \geq 4$ and $D_1.M_1 \geq \phi(M_1) \geq 2$. But then (\ref{eq:8.26'}) implies $D_1^2 = 0$, $M_1^2 = 4$ and $D_1.M_1 = 2$, which implies $D_1.M = D_1.M_1 = 2$, a contradiction since $D_1 \geq D + K_S$ and (\ref{eq:8.26}) imply $D_1.M \geq D.M = 3$.

{\bf Case (1B)($\beta$):} By (\ref{eq:8.36}) we have that $c_1(\F)^2 - 4c_2(\F) \geq - 2$, whence, as in the proof of \cite[Prop.3.1]{kl1}, there are two line bundles $N$ and $N'$ and a zero-dimensional subscheme $X \subset S$ fitting in an exact sequence
\begin{equation}
\label{eq:dm}
0 \hpil N \hpil \F \hpil \I_X  \* N' \hpil 0
\end{equation}
with $M \sim N + N'$ and $N'$ is base-component free and nontrivial. Moreover, again as in the proof of \cite[Prop.3.1]{kl1}, it can be easily deduced that two cases are possible: (i) $c_1(\F)^2 - 4c_2(\F) \geq 0$ and either $N \geq N'$ or $X = \emptyset$, \eqref{eq:dm} splits and also $N$ is base-component free and nontrivial; (ii) $c_1(\F)^2 - 4c_2(\F) = -2$ and $N'.N = c_2(\F), X = \emptyset$ and for every $\Delta \geq 0$ such that $h^0(\F(-\Delta)) > 0$ we can choose $N \geq \Delta$.
 
If we are in case (ii) then, by (\ref{eq:8.36}) and (\ref{eq:8.25}), we must have $D.M = D.L = 4 \; \mbox{and} \; l = 0$
and it follows that
\begin{equation}
\label{eq:8.260'}
M^2 = L^2 - 2D.L = 6, \; c_2(\F) = 2, \; M.L = M.D + M^2 = 10 \; \mbox{and} \; \phi(M) = 2,
\end{equation}
for if $\phi(M) = 1$, then, by \cite[Lemma2.4]{kl1}, we can write $M \sim 3F_1 + F_2$ with $F_i > 0$, $F_i^2 = 0$ and $F_1.F_2 = 1$, so that $M.L \geq 4\phi(L) = 12$, a contradiction.

\begin{claim}
\label{cl:9.1'}
There exists a divisor $F > 0$ such that $F^2 = 0$, $F.M = 2$ and $h^0(\F(- F)) > 0$.
\end{claim}
\begin{proof}
By \cite[Lemma2.4]{kl1} we can write $L \sim 2E + E_1 + E_2$ with $E_i > 0$, $E_i^2 = 0$, $i = 1, 2$ and $E.E_1 = 2$, $E.E_2 = E_1.E_2 = 1$. Moreover for any $E' > 0$ such that $E'.L = 3$ we have that either $E' \eqv E$ or $E' \eqv E_2$. Now $M^2 = 6$, $\phi(M) = 2$ hence
by \cite[Lemma2.4]{kl1} we can write $M \sim F_1 + F_2 + F_3$ with $F_i > 0$, $F_i^2 = 0$, and $F_i.F_j = 1$ for $1 \leq i < j \leq 3$. Since $10 = L.M = L.F_1 + L.F_2 + L.F_3 \geq 3 \phi(L) = 9$ we can assume that $L.F_1 = L.F_2 = 3$ hence that $F_1 \eqv E$, $F_2 \eqv E_2$. For any $F' \eqv F_2$ we have that $F' > 0$, $(F')^2 = 0$, $F'.M = 2$, $F'.L = 3$. Now  $3 = (F'+K_S).L < 2\phi(L)$, whence $h^1(\O_S(F'+K_S))= h^1(\O_S(- F')) = 0$ and, as in (\ref{eq:8.5}), we deduce that $h^0(\E(- F')) > 0$. Moreover (\ref{eq:E8}) with $\tau = \emptyset$ gives
\[0 \hpil \O_S(D - F') \hpil \E(C,A)(- F') \hpil \F (- F') \hpil 0. \]
Hence we will be done if we prove that either $h^0(D - F_2) = 0$ or $h^0(D - F_2 + K_S) = 0$. Suppose therefore that $h^0(D - F_2) > 0$, so that we just need to prove that $\Gamma := D - F_2$ is a nodal cycle. Note that $L.\Gamma = L.D - L.F_2 = L.(L - M) - 3 = 1$. Now if $\Gamma \sim A + \Delta$ with $A > 0$, $\Delta \geq 0$ and $A^2 \geq 0$ then $1 = L.\Gamma = L.A + L.\Delta \geq \phi(L) = 3$, a contradiction.
\end{proof}

{\it Continuation of the proof of Proposition} \ref{g^2_6particolari}.
By Claim \ref{cl:9.1'} and \eqref{eq:8.260'} we have that $N \geq F$, $X = \emptyset$ and $N'.N = 2$. Since $\tau = X=
\emptyset$, we see from \eqref{eq:eca} and (\ref{eq:E8}) that $\Bs |N'| \sub C$. If $\phi(N') = 1$, then $\Bs |N'| = \{x, y \}$, and since $C$ is general in its linear system, $x ,y \not\in C$. Hence $\phi(N') \geq 2$.
In particular $(N')^2 \geq 4$ so that $M^2 = N^2 + (N')^2 + 2N.N' \geq N^2 + 8$, which combined with (\ref{eq:8.260'}) yields $N^2 \leq -2$, whence the contradiction $2 = F.M \leq N.M = N^2 + N.N' \leq 0$.
Therefore, we must be in case (i). Since $N'$ is base-component free and nontrivial we must have $N'.L \geq 2\phi(L) = 6$. Since either $N \geq N'$ or $N$ is base-component free and nontrivial, we have $N.L \geq 6$ as well. But $L \sim D + N + N'$ then implies $14 = L^2 = L.D + L.N + L.N' \geq L.E + 12 = 15$, a contradiction.

This concludes case (1B)($\beta$) and the proof of Proposition \ref{g^2_6particolari} when $(L^2, \phi(L)) = (14,3)$.

{\bf Case 2: $(L^2, \phi(L)) = (16, 2)$.}

Take a nef $E > 0$ with $E.L = 2$ and $E^2 = 0$. Then from  \eqref{eq:eca}, using $h^1(-2E) = 0$, we get
\[ h^0(\E( -2E)) = h^0(\omega_C - A - (2E + K_S)_{|C}) \geq h^1(A) - 4 = 1, \]
so that there is a nonzero section in $H^0(\E)$ which vanishes along an element in $|2E|$. By Lemma \ref{useful3} we get a sequence as in (\ref{eq:E8}) with $D \geq 2E$, whence with $L.D \geq 4$. Using the same  notation as in that lemma, we have
\begin{equation}
\label{eq:9.10}
6 = c_2(\E) = D.M + c_2(\F) + l,
\end{equation}
where $l:= \length(\tau)$. Moreover we have
\begin{equation}
\label{eq:9bs}
\Bs|M| \sub \Supp(\tau) \cup \Bs|N_{C/S} - A| \sub \Supp(\tau) \cup C.
\end{equation}

\begin{claim}
\label{cl:9.1}
If $M^2 = 0$, then $M \sim P$ for an elliptic pencil $P$ and $c_2(\F) = 2$.
\end{claim}
\begin{proof}
We have $M \sim mP$ for an elliptic pencil $P$ and an integer $m \geq 1$. We have $P.D = P.L \geq 2\phi(L) = 4$, whence from (\ref{eq:9.10}) and \cite[Prop.3.2]{kn2} we get
$6 = D.M + c_2(\F) + l \geq mP.D - 2m + 4 \geq 2m + 4$,
which implies $m = 1$ and $c_2(\F) = 2$, as stated.
\end{proof}

{\it Continuation of the proof of Proposition} \ref{g^2_6particolari}.
Combining Claim \ref{cl:9.1} with \cite[Prop.3.2]{kn2} for $M^2 >0$ we get, in any case,
\begin{equation}
\label{eq.new}
c_2(\F) \geq 2.
\end{equation}
Combining with (\ref{eq:9.10}) we get
\begin{equation}
\label{eq:9.15}
D.M \leq 4 - l \leq 4 \; \mbox{and} \; D^2 = D.(L - M) \geq 0.
\end{equation}

\begin{claim}
\label{cl:9.2}
If $D^2 = 0$, then $D \sim 2E$.
\end{claim}
\begin{proof}
We have $D \geq 2E$. Assume that $D^2 = 0$ and $D \sim 2E + \Delta$ for some $\Delta > 0$. Then $\Delta^2 = - 4E.\Delta$. Moreover we have $4 = 2E.L = D.L - \Delta.L = D.M - \Delta.L \leq 4 - \Delta.L$ by (\ref{eq:9.15}). Therefore $D.M = 4$ and $\Delta.L = 0$, in particular $\Delta^2 < 0$,
so that $E.\Delta > 0$. By (\ref{eq:9.15}) again we must have $\tau = \emptyset$. Therefore $\Bs |M| \sub C$, but since $C$ is general in its linear system, it cannot contain any of the possible base points of $M$, whence $\phi(M) \geq 2$. Therefore $\Delta.M = D.M - 2E.M \leq 0$, which means that
$E.M = 2$ and $\Delta.M = 0$. But then we get the contradiction
$\Delta.L = 2E.\Delta + \Delta^2 + \Delta.M = -2E.\Delta < 0$.
\end{proof}

\begin{claim}
\label{cl:9.3}
We have $h^1(- D) = 0$.
\end{claim}
\begin{proof}
By Claim \ref{cl:9.2} we can assume $D^2 > 0$. By \cite[Cor.2.5]{klvan}, if $h^1(- D) = h^1(D + K_S) > 0$, there exists a $\Delta > 0$ with $\Delta^2 = -2$ and $\Delta.D \leq -2$. Since $L$ is nef we must have $\Delta.M \geq 2$, whence $(M + \Delta)^2 \geq 2$ and $0 \leq (M + \Delta).D \leq 2$ by 
\cite[Lemma2.1]{klvan} and (\ref{eq:9.15}). By the Hodge index theorem we have
$4 \leq  (M + \Delta)^2 D^2 \leq ((M + \Delta).D)^2 \leq 4$,
whence we must have $\Delta.D = -2, \; \Delta.M = 2, \; M^2 = 0, \; D.M = 4, \; D^2 = 2$,
and $D \eqv M + \Delta$. But then $\Delta.D = \Delta.M + \Delta^2 = 0$, a contradiction.
\end{proof}

{\it Conclusion of the proof of Proposition} \ref{g^2_6particolari}.
From (\ref{eq:8se}) and Claim \ref{cl:9.3} we find $h^0(\O_C(M)) = h^0(M)$ and since $\O_C(M) \geq A$ by Lemma \ref{useful3}, we must have $h^0(M) \geq 3$. Hence, using Claim \ref{cl:9.1}, we find that $M^2 \geq 4$.

If $D^2 > 0$ we have $E.D \geq 1$, and since $2 = E.L = E.D + E.M$ we must have $E.D = E.M = 1$, whence $|M|$ has two base points. Since $C$ is general in its linear system it cannot contain any of these, whence $l \geq 2$, and from (\ref{eq:9.15}) we get $D.M \leq 2$. But this is impossible by the Hodge index theorem.

Hence $D^2 = 0$ and by Claim \ref{cl:9.2} we have $D \sim 2E$, whence $D.L = D.M = 4$, so that by (\ref{eq:9.10}) and \eqref{eq.new}, we have $\tau = \emptyset$ and $c_2(\F) = 2$. Moreover we have $M^2 = L^2 - 2L.D = 8$, whence $c_1(\F)^2 - 4c_2(\F) = 0$. Therefore we have an exact sequence
\begin{equation}
\label{eq:seqbog3}
0 \hpil N \hpil \F \hpil \I_X \* N' \hpil 0,
\end{equation}
as in the proof of \cite[Prop.3.1]{kl1}, with $M \sim N + N'$, $N'$ is base-component free and nontrivial and either $N \geq N'$ or $N$ is base-component free and nontrivial. Therefore, in any case, $M.N \geq 2 \phi(M)$. Moreover, from (\ref{eq:seqbog3}), we find that
\begin{equation}
\label{eq:9.25}
2 = c_2(\F) = N.N' + \length(X) \geq N.N'.
\end{equation}

Now we claim that $\phi(M) = 2$.
Indeed $1 \leq \phi(M) \leq \lfloor \sqrt{8} \rfloor = 2$. If $\phi(M) = 1$ then $|M|$ has two base points, and since $\tau = \emptyset$ we get by (\ref{eq:9bs}) that $C$ contains these two base points, a contradiction on the generality of $C$ in its linear system.

As $h^0(N') \geq 2$ we must have $8 = M^2 = M.N + M.N' \geq 4\phi(M) = 8$,
which implies $M.N = M.N'= 4$. Combined with (\ref{eq:9.25}) and the Hodge index theorem we find that $N \eqv N'$ and $N^2 = 2$. In particular, by (\ref{eq:9.25}) again $X = \emptyset$. But then $|N'|$ has two base points, which have to be contained in $\Bs|N_{C/S} - A| \sub C$, again a contradiction.

This concludes the the proof of Proposition \ref{g^2_6particolari} when $(L^2, \phi(L)) = (16,2)$.
\end{proof}
\renewcommand{\proofname}{Proof}

\section{Complete intersections of two cubics}
\label{sec:r=3}

In this section we will prove that there are no exceptional curves of Clifford dimension $3$ on an Enriques surface.

Assume, to get a contradiction, that $C$ is an exceptional curve of Clifford dimension  $3$ lying on an Enriques surface $S$. By \cite[Satz1]{ma} $C$ is isomorphic to a complete intersection of two cubics, has genus $10$, Clifford index $3$ and possesses a unique line bundle $A$ computing its Clifford dimension, that is with $\dim |A|=3$ and $\deg A=9$. Also $A$ is very ample, it satisfies $\omega_C \sim 2A$, and embeds $C$ into $\PP^3$ as a complete intersection of two cubics. Moreover $C$ has gonality $6$, it has a $1$-dimensional family of $g^1_6$'s, and every $g^1_6$ is of the form $A-Z_3$, with $Z_3$ effective and $\deg Z_3=3$ \cite[Thm.3.7]{elms}. 

Set $L = \O_S(C)$. Then $L^2 = 18$. As $\gon C=6$ and $\phi(L) \leq \sqrt{L^2}$, we have $\phi(L)=3$ or $4$.

\subsection{The case $L \sim 3B$ with $B^2=2$}

Set $\E =\E(C,A)$.

\begin{claim}
\label{cl:ecc10}
We have $h^0(\E(-B)) >0$ and $h^0(\E(-B-K_S))>0$.
\end{claim}

\begin{proof}
 Since $B.L=6$, we have that $|\O_C(B)|$ is a $g^1_6$ on $C$, whence we have $\O_C(B) \sim A -Z_3$, for some effective $Z_3 \subset C$ of degree $3$.
Therefore 
\[ \N_{C/S}-A-\O_C(B+K_S) \sim \omega_C-A-\O_C(B) \sim A - (A-Z_3) = \O_C(Z_3) >0. \]
Tensoring \eqref{eq:eca} by $\O_S(-B+K_S)$, and using the fact that $h^1(-B+K_S)=0$ we find
$h^0(\E(-B-K_S)) = h^0(\N_{C/S}-A-\O_C(B+K_S)) >0$.
Similarly, $h^0(\E(-B)) >0$. 
\end{proof}
Pick a section $s \in H^0(\E)$ vanishing along some element of $|B|$ and denote by $D$ the largest effective divisor on which it vanishes. 
Then by Lemma \ref{useful3} we have an exact sequence
\begin{equation}
\label{eq:ecc10}
0 \hpil \O_S(D) \hpil \E \hpil \F \hpil \tau \hpil 0,
\end{equation}
where $\F$ is a locally free rank $3$ sheaf which is globally generated off a finite set contained in $C \cup \Supp \tau$, $\tau$ is a torsion sheaf supported on a finite set and $M:= \det \F$ is nontrivial, base-component free and $L \sim M+D$. 

\begin{claim} 
\label{cl:ecc11}
We have $D \sim B$, $M \sim 2B$ and $(c_2(\F), \length(\tau))=(5,0)$ or $(4,1)$.
\end{claim}

\begin{proof}
By Lemma \ref{bigger0} we also have 
\begin{equation}
\label{eq:ecc8}
D^2 >0, \; M^2 >0, \; D.M \leq 5 \; \mbox{and} \; 4 \leq c_2(\F) \leq 6.
\end{equation}

We now observe that $h^1(D+K_S)=0$. Indeed, if not, by \cite[Cor.2.5]{klvan}, there would exist a $\Delta >0$ satisfying $\Delta^2=-2$ and $\Delta.D \leq -2$. Since $L$ is nef we must have $\Delta.M \geq 2$. Hence $(D-\Delta)^2 \geq 4$ and $(D-\Delta).M \leq 3$ using \eqref{eq:ecc8}. Now if $h^0(D - \Delta) = 0$ then, by Riemann-Roch, $K_S - D + \Delta \geq 0$ and \cite[Lemma2.1]{klvan} implies that $0 \leq D.(K_S - D + \Delta) = - D^2 + D.\Delta \leq -4$, a contradiction. Therefore $D - \Delta \geq 0$, $M.(D - \Delta) \geq 0$ and the Hodge index theorem implies $(D-\Delta)^2=4$, $M^2=2$ and $(D-\Delta).M=3$,
whence $D^2=M^2=2$ and $D.M=5$, yielding the contradiction $18=L^2=(D+M)^2=14$.
From
\[ 0 \hpil \O_S(-D) \hpil \O_S(M) \hpil \O_C(M) \hpil 0 \]
we therefore find $h^0(\O_C(M))=h^0(M)$ and since $\O_C(M) \geq A$ by Lemma \ref{useful3}, we must have $h^0(M) \geq 4$, whence $M^2 \geq 6$ by Riemann-Roch and the nefness of $M$. The Hodge index theorem, \eqref{eq:ecc8} and the fact that
$D^2+M^2+2D.M=18$ now imply that $D^2=2$. 

If $M^2=6$, then $M.L=11$ and since $\O_C(M) \geq A$ with $h^0(\O_C(M))=h^0(A)$,
this means that $|M|$ has two base points, whence $\phi(M)=1$ and by \cite[Lemma2.4]{kl1} we can write
$M \sim 3E_1+E_2$ with $E_i >0$, $E_i^2=0$ and $E_1.E_2=1$. But then $M.L \geq 4\phi(L) \geq 12$, a contradiction.
Hence $M^2 \geq 8$, which implies $D.M=4$, so that $M \eqv 2D$ by the Hodge index theorem. Therefore $L \eqv 3D$, so that $M \sim 2B$.
By \eqref{eq:E11}  in Lemma \ref{useful3} we get $\length(\tau) + c_2(\F)=5$, and \eqref{eq:ecc8} yields the two stated possibilities.
\end{proof}
From \eqref{eq:ecc10} tensored by $\O_S(-B+K_S)$ we find, using Claim \ref{cl:ecc10},
\[ h^0(\F(-B+K_S)) \geq h^0(\E(-B+K_S))- h^0(K_S)=h^0(\E(-B+K_S)) >0. \]
Hence there is a section $t \in H^0(\F)$ vanishing along some element of $|B+K_S|$. Denoting by $D_1$ the largest effective divisor on which it vanishes, we get as above an exact sequence
\begin{equation}
\label{eq:ecc13}
0 \hpil \O_S(D_1) \hpil \F \hpil \G \hpil \tau_1 \hpil 0,
\end{equation}
where $\G$ is a locally free rank $2$ sheaf which is globally generated off a finite set contained in $C \cup \Supp \tau \cup \Supp \tau_1$, $\tau_1$ is a torsion sheaf supported on a finite set and $M_1:= \det \G$ is nontrivial and base-component free and $M \sim 2B \sim M_1+D_1$ (look at the proof of Lemma \ref{useful3}; the fact that $M_1$ is nontrivial follows by Porteous \label{port} (see for example \cite[proof of Lemma 1.12]{par} since $\G$ is nontrivial because, as is well-known, $H^2(\E(K_S)) = 0$, whence also $H^2(\G(K_S)) = 0$). Taking $c_2$ in \eqref{eq:ecc13} and combining with Claim \ref{cl:ecc11} we find
\begin{equation}
\label{eq:ecc12}
5 = D_1.M_1+ \length(\tau) + \length(\tau_1) + c_2(\G).
\end{equation}

If $M_1^2=0$ then $M_1 \sim kP$ for an integer $k \geq 1$ and  an elliptic pencil $P$, and $c_2(\G) \geq 4-2k$ 
by \cite[Prop.3.2]{kn2}. By \eqref{eq:ecc12} we have, using $\phi(M)=2$, the contradiction
\[ 5 \geq kP.D_1 +c_2(\G) = kP.M +c_2(\G) \geq 4k +4-2k = 2k+4 \geq 6. \]
Hence $M_1^2 >0$, so that $c_2(\G) \geq 2$ by \cite[Prop.3.2]{kn2}. Therefore $D_1.M_1 \leq 3$ by \eqref{eq:ecc12}. 
Moreover we have $D_1.M \geq B.M=4$, and since $D_1.M=D_1^2+D_1.M_1 \leq D_1^2+3$, we have 
$D_1^2 \geq 2$. Now the Hodge index theorem and the fact that $D_1^2+M_1^2+2D_1.M_1=8$ imply that $D_1 \eqv M_1$ and $D_1^2=2$. Therefore $M \sim 2B \eqv 2D_1$ and it follows that $D_1 \sim M_1 \sim B+K_S$. Now $|M_1|$ has two 
base points, and by the above they must lie in $C \cup \Supp \tau \cup \Supp \tau_1$. From  \eqref{eq:ecc12} we 
get $\length(\tau) + \length(\tau_1) \leq 1$, so that at least one of the base points of $|M_1|$, say $x$, lies on $C$. As $H^1(M_1-L)=0, |\O_C(M_1) -x|$ is a $g^1_5$, a contradiction.

\subsection{The case $L$ not divisible by $3$ in $\Pic S$} 

Let $Z \in |A|$. Recall that $\O_C(2Z) \sim \omega_C$.

From 
\begin{equation*}
0 \hpil K_S \hpil \I_Z(L+K_S)  \hpil A \hpil  0
\end{equation*}
we get
\begin{equation}
\label{eq:c1}
h^0(\I_Z(L+K_S))=h^0(A)=4
\end{equation}
and
\begin{equation}
\label{eq:c2}
\Bs |\I_Z(L+K_S)| \cap C \sub Z.
\end{equation}

We can write
\begin{equation}
\label{eq:c4'}
|\I_Z(L+K_S)| = \{M \} + \Delta, 
\end{equation}
where $\{M \}$ is the moving part which is a sublinear system   $\{M \} \sub |M|$ for some
$M$ which is without fixed components (whence nef) and $\Delta$ is the fixed divisor. Clearly
\begin{equation}
\label{eq:c5}
h^0(M) \geq \dim \{M \}+1 = h^0(\I_Z(L+K_S))=4.
\end{equation}
  
We will now use the set $\Pi$ in \eqref{eq:pi} (see also \eqref{eq:Q}).

\begin{lemma} 
\label{lemmaA}
If $Z$ does not intersect $\Pi$ then $\Delta.L=0$.
\end{lemma}

\begin{proof}
Assume to get a contradiction that $\Delta.L>0$. By \eqref{eq:c2} we have $ \Delta \cap C \sub Z$, so that, by definition of $\Pi$, $h^0(\Delta) \geq 2$, and in particular
\begin{equation}
\label{eq:c4}
6 \leq 2\phi(L) \leq \Delta.L \leq \deg Z =9.
\end{equation}

If $M^2=0$ then $M \sim mP$ for an elliptic pencil $|P|$ and an integer $m \geq 3$, by (\ref{eq:c5}). But then we get the contradiction $L^2 = \Delta.L +M.L \geq 2\phi(L) +2m\phi(L) \geq 8 \phi(L) \geq 24$. 

Hence $M^2 >0$ and from (\ref{eq:c5}) we must even have $M^2 \geq 6$. Consequently $M.L \geq 3\phi(L)$ and as above $18=L^2= \Delta.L +M.L \geq 2\phi(L)+3\phi(L)=5\phi(L)$ implies $\phi(L)=3$.

We now claim that $\Delta^2 \geq 2$. Indeed if  $\Delta^2 \leq 0$, then (\ref{eq:c4}) yields
$\Delta.M = \Delta.L - \Delta^2 \geq 6$ and since $M^2 \geq 6$ we then get $L^2=18=M.L+\Delta.L = M^2 + M.\Delta + \Delta.L \geq 18$, whence we must have equalities all the way, that is $M^2=M.\Delta=\Delta.L=6$ and $\Delta^2=0$. Then we can write
$Z = (C \cap \Delta) \cup Z_3$
with $\deg Z_3=3$. It follows that $\{M \} \sub |\I_{Z_3} \* M| \sub |M|$ and $3=\dim \{M \}= \dim |M|$, whence $|M|$ has a base scheme of length three, a contradiction. This shows that $\Delta^2 \geq 2$.
The case $\Delta^2 \geq 4$ is easily ruled out by the Hodge index theorem, whence $\Delta^2=2$. The Hodge index theorem yields $M.\Delta \geq 4$ and one easily sees that $L^2=18$ yields only the two possibilities $(\Delta^2, M^2, M.\Delta)=(2,6,5)$ or $(2,8,4)$. 
In the first case we can write $Z = (C \cap \Delta) \cup Z_2$
with $\deg Z_2=2$. It follows that $\{M \} \sub |\I_{Z_2} \* M| \sub |M|$ and $3=\dim \{M \}= \dim |M|$, whence $Z_2$ consists of the two base points of $|M|$. It is well-known that such base points are contained in halfpencils, contradicting our choice of $Z$. 
Hence we must be in the second case, where the Hodge index theorem yields $M \eqv 2\Delta$ so that $L \eqv 3\Delta$, contradicting our assumptions.
\end{proof}
We will from now on fix once and for all a $Z \in |A|$ subject to the following four conditions:
\begin{itemize}
\item[(C1)] $Z$ consists of nine distinct points outside of $\Pi$.
\item[(C2)] No six of the nine points in $Z$ form a $g^1_6$ on $C$.
\item[(C3)] If $|P|$ is a complete base-component free pencil on $S$, then for any $P_0 \in |P|$ we have 
$\length(P_0 \cap Z) \leq 1$ and for any singular or reducible $P_0 \in |P|$ we have $P_0 \cap Z =\emptyset$. 
\item[(C4)] If $|B|$ is a complete base-component free net on $S$, then $\length(B_0 \cap Z) \leq 1$ for any 
singular $B_0 \in |B|$, and if furthermore $|B|$ is nonhyperelliptic, then $\length(B_0 \cap Z) \leq 1$ for any smooth hyperelliptic curve $B_0 \in |B|$.
\end{itemize}

Clearly, since $A$ is base-point free and the curves in $\Pi$ are finite, the general $Z \in |A|$ satisfies (C1).
Moreover, since the family of $g^1_6$'s on $C$ has dimension one, we see that the dimension of the family consisting of $Z \in |A|$ such that a length six subscheme forms a  $g^1_6$ is at most $2 < \dim |A|=3$, whence 
the general $Z \in |A|$ satsifies (C2).

To see that the general $Z \in |A|$ satisfies (C3), consider a base-component free pencil $|P|$ on $S$ and let 
$P_0 \in |P|$. Note that $P_0 \cap Z$ consists of distinct points, so that it has only finitely many subschemes. If $Z_2 \sub P_0 \cap Z$ is any subscheme of length two, then $\dim |A-Z_2| = \dim A-2=1$, as $|A|$ is very ample. Consider $\J \sub |P| \x |A|$ given by
\[ \J : = \{ (P_0,Z) \; | \; P_0 \in |P|, \; Z \in |A| \; \mbox{and} \; 
\length(P_0 \cap Z) \geq 2 \} \]
and denote by $\pi_1$ its projection to $|P|$ and $\pi_2$ its projection to $|A|$. Then, by what we saw right above, 
$\dim \pi_1^{-1}(P_0) = \dim |A|-2=1$ for any $P_0 \in |P|$, so that $\dim \pi_2(\J) \leq 1+ \dim |P| = 2=
\dim |A|-1$, so that $\pi_2$ is not surjective and the general $Z \in |A|$ satisfies $\length(P_0 \cap Z) \leq 1$ 
for any $P_0 \in |P|$. Moreover, since the singular and reducible members of $|P|$ are a finite number, the general 
$Z \in |A|$ also satisfies $P_0 \cap Z =\emptyset$ for every singular or reducible $P_0 \in |P|$. Therefore, as the possible $P$'s are countably many, the general $Z \in |A|$ satisfies (C3).
Similarly, since the family of singular curves in a complete base-component free net on $S$, and the family of smooth 
hyperelliptic curves in a complete base-component free nonhyperelliptic net on $S$, both have dimension one, we can argue as 
above, substituting $|P|$ with any irreducible family of dimension one of hyperelliptic smooth curves
or singular curves in the net $|B|$, and prove that the general $Z \in |A|$ satisfies (C4).

This shows that we can indeed choose a $Z \in |A|$ satisfying (C1)-(C4).

From Lemma \ref{lemmaA} and property (C1) we get that we can write (\ref{eq:c4'}) as
\begin{equation*}
|\I_Z(L + K_S)| = |\I_Z \* M| + \Delta, \hs \Delta \cap C =\emptyset.
\end{equation*}
Moreover, using that $\Delta^2 \leq -2$ by Lemma \ref{lemmaA}, we have
$M^2= L^2 + \Delta^2 = 18 + \Delta^2 \leq 18$.
Since $M_{|C}-A \sim (L+K_S-\Delta)_{|C}-A \sim \omega_C -A \sim A$ and
$h^0(\I_Z \* M)=h^0(\I_Z(L+K_S))=h^0(A)=4$
from (\ref{eq:c1}), we get that the natural restriction map arising from
\begin{equation*}
0 \hpil K_S-\Delta \hpil \I_Z \* M \hpil A \hpil  0 
\end{equation*}
is an isomorphism:
$\xymatrix{ H^0 (\I_Z \* M) \ar[r]^{\ \ \cong} & H^0(A) }$.
Moreover, from (\ref{eq:c2}), we have
\begin{equation}
\label{eq:c14}
\Bs |\I_Z \* M| \cap C = Z.
\end{equation}
Set $\overline{Z} =\Bs |\I_Z \* M| \sup Z$ and denote by $f: \tilde{S} \hpil S$
the resolution of $\overline{Z}$. Let $\tilde{H}$ be the strict transform of the general element in $|\I_Z \* M|$ and let $\tilde{C}$ be the strict transform of $C$. Then $|\tilde{H}|$ is base-point free with $h^0(\tilde{H})=4$ and since $\deg \overline{Z} \geq 9$ we must have
\begin{equation}
\label{eq:c20}
\tilde{H}^2 = 9 -k, \hs 0 \leq k \leq 9.
\end{equation}

\begin{lemma} 
\label{lemmaB}
If the general curve in $|\I_Z \* M|$ is singular, then $k \geq 4$.
\end{lemma}

\begin{proof}
If the general curve in $|\I_Z \* M|$ has a point $x$ of multiplicity $\geq 2$ then, by Bertini's theorem, $x$ is a base point of $|\I_Z \* M|$. Now one easily sees that resolving the base scheme located at $x$
makes the self-intersection drop at least by $4$.
\end{proof}

Now let $\varphi = \varphi_{\tilde{H}}$ be the morphism to $\PP^3$ defined by ${\tilde{H}}$ and denote by $S_0$ the image of $\tilde{S}$ and $C_0$ the image of $\tilde{C}$. Then by construction
$C_0$ is the Clifford embedding of $C$, that is letting $\varphi_A$ be the morphism defined by $A$ we have a commutative diagram
\[ \xymatrix{ \tilde{C} \ar_{\varphi_{|\tilde{C}}}[rd] \ar^{f_{|\tilde{C}}}[r]  &  C \ar^{\varphi_A}[d] \\ & C_0. } \]
 
In particular $C_0$ is smooth and nondegenerate in $\PP^3$ of degree $9$ and is the complete intersection of two cubics. 

If $\dim S_0=1$ then $\varphi$ is composed with a rational pencil (since $h^1(\O_{\tilde{S}})=h^1(\O_S)=0$), so that $\varphi$ factorizes as $\tilde{S} \khpil \PP^1 \khpil \PP^3$ and $C_0$ is the twisted cubic in $\PP^3$, a contradiction.

Hence $S_0$ is a surface and, since $\deg C_0=9$, we have $\tilde{C}. {\tilde{H}}=9$.

Moreover, by (\ref{eq:c14}) we have 
$\tilde{C}^2=\tilde{C}.K_{\tilde{S}}=9$.
Set $d = \deg \varphi$ and $d_0 = \deg S_0$.
Then
\begin{equation}
\label{eq:c25}
d d_0 = 9-k, \hs d_0 \geq 3, \hs 0 \leq k \leq 6, 
\end{equation}
where we have used (\ref{eq:c20}) and the fact that $C_0$ is neither contained in any hyperplane nor in any quadric to conclude that $d_0 \geq 3$.

Now $S_0$ is Cartier in $\PP^3$ whence it is Cohen-Macaulay  and by adjunction
\begin{equation}
\label{eq:c26}
\omega_{S_0} \iso \omega_{\PP^3} (S_0) \* \O_{S_0} \iso \O_{S_0}(d_0-4).
\end{equation}
Let $ \Theta = |\I_{C_0 / \PP^3} (3)|$, which has dimension one. Since $C_0$ is the complete intersection of two cubics, $\Theta$ is base-point free off $C_0$, and, since $C_0$ is smooth and Cartier on any member of $\Theta$, the general member of  $\Theta$ is smooth by Bertini's theorem.

Now we choose once and for all a smooth irreducible surface $T \in \Theta$ such that
\begin{equation*}
T \cap S_0 = C_0 \cup C_1,
\end{equation*}
where $C_1$, if not empty, is an irreducible curve such that
\begin{eqnarray}
\label{eq:c33a} 
C_0 \ \mbox{and} \ C_1 \ \mbox{intersect transversally and outside of} \ \Sing S_0  \ \mbox{and the } \\
\nonumber \mbox{finitely many points on} \
S_0 \ \mbox{coming from the curves contracted by} \ \varphi, 
\end{eqnarray}
\begin{equation*}
C_1 \ \mbox{does not meet the isolated singularities of} \ S_0,
 \mbox{and} \ \Sing C_1 \sub \Sing S_0.
\end{equation*}
Also note that since $3d_0=\O_{S_0}(1).\O_{S_0}(3) = \deg C_0 + \deg C_1$ we have
\begin{equation}
\label{eq:c33d}
\deg C_1= 3(d_0-3).
\end{equation}

Now the surface $T \subset \PP^3$ satisfies $K_T = \O_T(-1)$ and $K_T^2=3$. On $T$ we have $T \cap S_0 \in |\O_T(d_0)|$ whence $C_0 + C_1 \in |O_T(d_0)|$. Furthermore, since $C_0 \in |\O_T(3)|$ (being the complete intersection of two 
cubics) we have 
\begin{equation*}
C_1 \sim \O_T(d_0-3) \sim -(d_0-3)K_T, 
\end{equation*}
whence on $T$ we have $C_0.C_1=3(d_0-3)K_T^2= 9(d_0-3)$. Since $C_0$ and $C_1$ intersect transversally and outside 
$\Sing S_0$ we have 
\begin{equation}
\label{eq:c33f}
\# (C_0 \cap C_1) = 9(d_0-3).
\end{equation}

We now note that if $d \geq 2$, then from (\ref{eq:c25}) the only possibilities are $(d, d_0, k)=(3,3,0)$, $(2,3,3)$ or $(2,4,1)$. 
We will now divide the rest of the proof into the three cases $d =1$, $(d, d_0, k)=(2,4,1)$ and $(d, d_0, k)=(3,3,0)$ or $(2,3,3)$. 

\subsubsection{The case $d =1$.}
We have
\begin{equation*}
3\tilde{H} \sim \varphi^* (C_0 + C_1) \sim \tilde{C} + \tilde{C_1} + \sum r_iR_i,
\end{equation*}
where $\tilde{C_1}$ is the strict transform of $C_1$ and the $R_i$ are the exceptional divisors of
$\varphi$. Define $\tilde{C_2} = \tilde{C_1} + \sum r_iR_i$, then $3\tilde{H} \sim \tilde{C} + 
\tilde{C_2}$ and $27=3\tilde{H}.\tilde{C}= \tilde{C}^2 + \tilde{C}.\tilde{C_2}$ implies
$18= \tilde{C}.\tilde{C_2}= \sum r_iR_i.\tilde{C}+ \tilde{C_1}.\tilde{C}$, whence $\tilde{C_1}.\tilde{C} \leq 18$, and it follows from (\ref{eq:c33a}) that
$\# (C_0 \cap C_1) \leq 18$.

Comparing with (\ref{eq:c33f}) we see that 
\begin{equation}
\label{eq:c35b}
\tilde{H}^2=d_0 =9-k \leq 5, \; \mbox{whence} \; k \geq 4.
\end{equation}
Now consider the Stein factorization of $\varphi$:
\[ \xymatrix{ \tilde{S} \ar^{\pi_1}[r] & \overline{S_0} \ar^{\pi_2}[r]  &  S_0. } \]
Then $\overline{S_0}$ is normal and, as $d = \deg \varphi = 1$, $\pi_2$ is an isomorphism, so that $S_0$ is normal and we can assume that $\pi_1 = \varphi$. Using (\ref{eq:c26}) we get
\begin{equation}
\label{eq:c29}
K_{\tilde{S}} \eqv \varphi^*(K_{S_0}) + \sum c_iR_i \eqv (d_0-4)\tilde{H} +\sum c_iR_i,
\end{equation}
for some $c_i \in \QQ$. Also note that since $\tilde{H}$ is nef and 
\begin{equation*}
K_{\tilde{S}} \sim f^*K_S + \sum a_i \mathfrak{e}_i, a_i \geq 0,
\end{equation*}
where the $\mathfrak{e}_i$ are the exceptional divisors of $f$, we have $\tilde{H}.K_{\tilde{S}} \geq 0$.

Since $h^2(\tilde{H})=h^0(K_{\tilde{S}}-\tilde{H})=0$, we have
$h^1(\tilde{H}) = h^0(\tilde{H})-\frac{1}{2}\tilde{H}.(\tilde{H}-K_{\tilde{S}})-1 \geq 1$.

Let $H \in |\tilde{H}|$ be a general smooth curve. From the short exact sequence
\[ 0 \hpil \O_{\tilde{S}} \hpil \O_{\tilde{S}}(\tilde{H}) \hpil \O_{H}(\tilde{H})  \hpil 0, \]
we see that $h^0(\O_{H}(\tilde{H})) =3$ and $h^1(\O_{H}(\tilde{H})) =h^1(\tilde{H}) \geq 1$. As
$|\O_{H}(\tilde{H})|$ is birational, it cannot be a multiple of a $g^1_2$ on $H$. If $\O_{H}(\tilde{H}) \cong \omega_H$ 
we get $\tilde{H}.K_{\tilde{S}}=0$ whence $\tilde{H}.\mathfrak{e}=0$ for every exceptional divisor $\mathfrak{e}$, so that $\tilde{H} = f^{\ast}(M_0)$ for some $M_0$, a contradiction. Now Clifford's theorem gives $0 < \Cliff \O_{H}(\tilde{H}) = \tilde{H}^2 - 4 = d_0-4$. 
Combining with \eqref{eq:c35b} we get
\begin{equation}
\label{eq:d2}
\tilde{H}^2=d_0=5 \; \mbox{and} \; k=4.  
\end{equation}
It follows from the proof of Lemma \ref{lemmaB} that the general curve in $|\I_Z \* M|$ has at most one singular 
point, and if so, it is of multiplicity two. 

Using Lemma \ref{lemmaA} and \eqref{eq:d2} we get $5 = \tilde{H}^2 \leq (L-\Delta)^2 -9 = 9+\Delta^2$,
whence $\Delta^2 \geq -4$. We also have
\[ M^2=18+\Delta^2 = 2p_a(M)-2 \leq 2p_a(\tilde{H})-2+2 = 
\tilde{H}.(\tilde{H}+K_{\tilde{S}})+2=7+\tilde{H}.K_{\tilde{S}}, \]
which yields
$\tilde{H}.K_{\tilde{S}} \geq 11 +\Delta^2 \geq 7$. Combined with (\ref{eq:c29}) we get the contradiction
\[ 7 \leq \tilde{H}.K_{\tilde{S}}= \tilde{H}. \Big(\tilde{H} + \sum c_iR_i \Big) = 5.  \]

\subsubsection{The case $(d, d_0, k)=(2,4,1)$.}

First of all note that we have $\tilde{H}^2=8$, whence, using Lemma \ref{lemmaA}, we get 
$8 = \tilde{H}^2 \leq (L - \Delta)^2 - 9 = 9 + \Delta^2$ so that $\Delta^2 \geq -1$, whence $\Delta=0$ and $M \sim L+K_S$. Moreover the general curve in $|\I_Z \* M|$ is smooth by Lemma \ref{lemmaB}.

We now show that $S_0$ is normal.

Assume, to get a contradiction, that $S_0$ is not normal. Since it is Cohen-Macaulay, it is singular in codimension one, so that the general smooth curve in $|\tilde{H}|$ is mapped $2:1$ to a hyperplane section of $S_0$, which is a singular curve of arithmetic genus $3$. This map factors through the normalisation of the hyperplane section, whence the general smooth curve in $|\tilde{H}|$ can be mapped $2:1$ to a smooth curve of genus $\leq 2$. It follows that the general smooth curve in $|\tilde{H}|$ has gonality $\leq 4$. Since the elements in an open dense subset of the smooth curves in $|\tilde{H}|$ are in one-to-one correspondence with the smooth curves in $|\I_Z \* M|$, it follows by \cite[Thm.1.4]{glm1} that the gonality is $4$ for the  general smooth curve in $|\I_Z \* M| \sub |M|$.
By Lemma \ref{lemma:18,1,4} right below, we have $M \sim L+K_S \eqv 3D$ with $D^2=2$, contradicting our assumptions.
Therefore we have shown that $S_0$ is normal.

We now prove the needed lemma, which will be useful later as well:

\begin{lemma} 
\label{lemma:18,1,4}
Let $N$ be a nef line bundle with $N^2=18$ and $\phi(N) \geq 3$ on an Enriques surface $S$. Assume $Y \in |N|$ is a smooth curve of gonality $4$. Then $N \eqv 3D$ with $D^2=2$.
\end{lemma}

\begin{proof}
By \cite[Prop.3.1]{kl1} there is a base-component free linear system $|D|$ such that $2D^2 \leq N.D \leq D^2+4 \leq 8$. Since $\phi(N)\geq 3$, we must have $N.D \geq 2\phi(N) \geq 6$, whence $D^2 \geq 2$. The Hodge index theorem yields $D^2=2$ and $N \eqv 3D$.
\end{proof}

Let $\sigma: \tilde{S_0} \khpil S_0$ be a minimal desingularisation. 
After taking a succession of monoidal transformations $f': \tilde{\tilde{S}} \khpil \tilde{S}$ we get a commutative diagram  
\[ 
\xymatrix{\tilde{\tilde{S}}  \ar^{f'}[r] \ar_{\Phi}[d]  &  \tilde{S} \ar^{\varphi}[d]  \ar^{f}[r] & S \\  
\tilde{S_0} \ar^{\sigma}[r]  & S_0} 
\]
with $\Phi$ a degree two morphism of smooth surfaces.

We now claim that $h^1(\O_{\tilde{S_0}})=0$.
Indeed consider the Stein factorization of $\Phi$:
\[ 
\xymatrix{
  \tilde{\tilde{S}} \ar^{\Phi_1}[r] &  {\tilde{S}_0}' \ar^{\Phi_2}[r] & \tilde{S_0}.  
} \]
Then $({\Phi_1})_*\O_{\tilde{\tilde{S}}} \iso \O_{{\tilde{S}_0}'}$, so that ${\tilde{S}_0}'$ is normal (see for example \cite[Prop.1.2.16]{mat}). By the lemma of Enriques-Severi-Zariski \cite[III, Thm. 7.8]{hart} and the criterion of \cite[III, Thm. 7.6(b)]{hart}, we have that ${\tilde{S}_0}'$ is Cohen-Macaulay. Now the double cover $\Phi_2$ satisfies $({\Phi_2})_*\O_{{\tilde{S}_0}'} \iso \O_{{\tilde{S}_0}} \+ \L$, for some $\L \in \Pic {\tilde{S}_0}$ \cite[Chp. 0, section 1]{cd}. Hence ${\Phi}_*\O_{\tilde{\tilde{S}}} \iso \O_{{\tilde{S}_0}} \+ \L$.
By Leray we get $h^1({\Phi}_*\O_{\tilde{\tilde{S}}}) \leq h^1(\O_{\tilde{\tilde{S}}})=h^1(\O_S)=0$. Hence $h^1(\O_{\tilde{S_0}})=0$, as claimed.  

As an immediate consequence, we get from combining \cite[Prop.3.7]{lau}, \cite[Prop.8 and Thm.1 on p.~345]{ume} and the fact that $\omega_{S_0} \iso \O_{S_0}$ by \eqref{eq:c26}, that $\tilde{S_0}$ must be a $K3$ surface.

Now $\Phi$ is ramified along an effective divisor $R$ (possibly zero) and we have
$K_{\tilde{\tilde{S}}} \sim \Phi ^* K_{\tilde{S_0}} +R \sim R$.
Therefore $h^0(K_{\tilde{\tilde{S}}})=h^0(R) >0$, which is impossible, since $h^0(K_{\tilde{\tilde{S}}})= h^2(\O_{\tilde{\tilde{S}}})=h^2(\O_S)=0$. 
Hence we have a contradiction, which rules out this case. 

\subsubsection{The cases $(d, d_0, k)=(3,3,0)$ or $(2,3,3)$}

By Lemma \ref{lemmaA} and \eqref{eq:d2} we have $\tilde{H}^2 \leq (L-\Delta)^2 - 9 = 9 + \Delta^2$, whence the three options
\begin{itemize}
\item[(a)] $\tilde{H}^2=9$, $\Delta=0$, $d=3$.
\item[(b)] $\tilde{H}^2=6$, $\Delta=0$, $d=2$.
\item[(c)] $\tilde{H}^2=6$, $M^2=16$, $\Delta^2=-2$, $d=2$.
\end{itemize}
Moreover, in this case $C_1=0$ by \eqref{eq:c33d}, whence $C_0 \in |\O_{S_0}(3)|$ is Cartier and $S_0$ has only isolated singularities, in particular it is normal. From (\ref{eq:c26}) we find that $\omega _{S_0} \iso \O_{S_0}(-1)$. Also note that, in all cases, the general curve in $|\I_Z \* M|$ is smooth by Lemma \ref{lemmaB}.

By \cite[Prop.0.3.3]{cd} $S_0$ is either an anticanonical del Pezzo surface or a projection of a  scroll or an elliptic cone. In the two latter cases the pullback by $\varphi$ of the ruling is a moving complete linear system $|R|$ on $\tilde{S}$ such that $R.\tilde{H}=d$, that is all the smooth curves in $|\tilde{H}|$ have gonalities $\leq d \leq 3$. As the general element in $|\I_Z \* M|$ is smooth, the curves in an open, dense subset of the smooth curves in $|\I_Z \* M|$ are in one-to-one correspondence with the smooth curves in $|\tilde{H}|$ and they all have gonalities $\leq d \leq 3$. Since $g(M)=9$ or $10$ it follows from \cite[Thm.1.4]{glm1} that they must all have gonality two, that is they are hyperelliptic. Therefore $\phi(M)=1$, so that $|M|$ has base points and we must be in case (c), since $L$ is base-point free. One easily sees that one can write $M \sim 8E_1+E_2$ with both $E_i >0$ and $E_i^2 =0$, such that $E_1.E_2=1$. Since $\phi(L)  \geq 3$ we must have $E_1. \Delta \geq 2$ and since $\Delta.M =2$ it follows that $\Delta.E_2\leq -14$, whence the contradiction $E_2.L =E_2.M +E_2.\Delta \leq -6$.

Note that we have also just proved that $\phi(M) \geq 2$, which will be useful later.

Hence $S_0$ is an anticanonical del Pezzo surface. 

We now rule out the cases (b) and (c). 

In these two cases we have $d=2$, so that the general smooth curve  in $|\tilde{H}|$ is mapped generically $2:1$ to a plane cubic. Therefore the general smooth curve in $|\tilde{H}|$ possesses a one dimensional family of complete $g^1_4$'s. 
In case (b), the smooth curves in an open, dense subset of the smooth curves in $|\I_Z \* M|$ are in one-to-one correspondence with the smooth curves in $|\tilde{H}|$, and by Lemma \ref{lemma:18,1,4} we have that $L$ is $3$-divisible, a contradiction.
As for case (c), we first need the following result:

\begin{lemma} 
\label{lemma:16,1,4}
Let $N$ be a nef line bundle with $N^2=16$ and $\phi(N) \geq 2$ on an Enriques surface $S$. Assume $Y \in |N|$ is a smooth curve of gonality $4$. Then any $g^1_4$ on $Y$ is of one of the following types:
\begin{itemize}
\item[(i)]  $\O_Y(P)$ for an elliptic pencil $P$ with $P.N=4$ (in particular $\phi(N)=2$); 
\item[(ii)] $\O_Y(D)-x-y$, for a base-component free linear system $|D|$ with $D^2=2$ and
             two distinct base points $x$ and $y$;
\item[(iii)] $\O_Y(D)-Z_4$, for a base-component free linear system $|D|$ with $D^2=4$ and $N \eqv 2D$ and $Z_4$ a $0$-dimensional scheme of length $4$ imposing only one condition on $|D|$.
\end{itemize}
\end{lemma}

\begin{proof}
Let $|B|$ be a base-point free complete $g^1_4$ on $Y$. By \cite[Prop.3.1]{kl1} we can write $N \sim D + N_1$ with $N_1 > 0, |D|$ base-component free and such that $2D^2 \leq N.D \leq D^2+4 \leq 8$ and $\O_Y(D) \geq B$. 
If $D^2=0$, then $D.N=4$ and $|D|$ is an elliptic pencil, since $\phi(N) \geq 2$, and we are in case (i).
If $D^2=2$ then $D.N=6$ by the Hodge index theorem, and we claim that $H^1(-N_1) = 0$. As $N_1^2 = 6$ if $H^1(-N_1) \neq 0$ by \cite[Cor.2.5]{klvan} there exists a $\Delta >0$ such that $\Delta^2=-2$ and $k:= - \Delta.N_1 \geq 2$. By \cite[Lemma2.3]{kl1} we can write $N_1 \sim A_1 + k \Delta$ with $A_1 > 0, A_1^2 = 6$. Now $D.A_1 \geq 4$ by the Hodge index theorem and $4 = D.N_1 = D.A_1 + k D.\Delta \geq 4$ gives $D.\Delta = 0$ whence the contradiction $0 \leq N.\Delta = D.\Delta + N_1.\Delta = -k$. Therefore $H^1(-N_1) = 0$ and $h^0(\O_Y(D))=h^0(D)=2$. Hence $B \sim \O_Y(D)-x-y$, where $x$ and $y$ are the two distinct base points of $|D|$. This is case (ii).
Finally, if $D^2=4$, then $N \eqv 2D$ by the Hodge index theorem and for reasons of degree $B \sim \O_Y(D) -Z_4$ for a $0$-dimensional scheme $Z_4$ of length $4$ on $Y$. As $h^1(D)=h^1(D+K_S)=0$ by the nefness of $D$, we must have 
$h^0(\O_Y(D))=h^0(D)=3$ and $h^0(B)=h^0(\O_Y(D) -Z_4)=h^0(\I_{Z_4}(D))=2$, we see that the scheme $Z_4$ poses only one condition on $H^0(D)$. This is case (iii).
\end{proof}

Returning to case (c), the curves in an open, dense subset of the smooth curves in $|\tilde{H}|$ are in one-to-one correspondence with the smooth curves in $|\I_Z \* M| \sub |M|$ and their $g^1_4$'s are therefore given by Lemma \ref{lemma:16,1,4}. 
Since the general smooth curve in $|\I_Z \* M|$ has infinitely many $g^1_4$'s, they cannot all be of type (i), that is cut out by an elliptic pencil, nor can they all be of type (ii), since there are only finitely many linear equivalence classes of divisors $D$.
Hence the general smooth curve $Y \in |\I_Z \* M|$ has infinitely many $g^1_4$'s of type (iii) in Lemma \ref{lemma:16,1,4}, that is $M \eqv 2D$  for a base-component free linear system $|D|$ with $D^2=4$ and these $g^1_4$'s are of type $\O_Y(D)-Z_4$ with $Z_4$ a $0$-dimensional scheme of length $4$ imposing only one condition on $|D|$. Since $D^2=4$, and $h^0(\I_{Z_4}(D))=2$, in general the scheme $Z_4$ is uniquely determined by any of the four points in $Z_4$, and for general $x \in S$ the $Z_4$ corresponding to $x$ is nothing but
\[ Z_4(x) := \Bs |\I_x(D)|. \]
Since the general smooth curve in $|\I_Z \* M|$ has a one dimensional family of $g^1_4$'s of this type, we must have that the general element in $|\I_Z \* M|$ passing through $x$ also passes through the whole of $Z_4(x)$.
But this implies that the map given by $|\tilde{H}|$ is $4:1$, a contradiction. 

We have therefore shown that we are in case (a), that is $\tilde{H}^2=9$, $\Delta=0$ and $d=3$.

We now treat this case. Note that $\overline{Z} = Z$ and $f$ is the blow-up at nine distinct points.

\begin{claim} 
\label{cl:eecc1}
Let $|P|$ be a base-component free pencil with $(f^*P).\tilde{H}=6$ or $8$.
Then $|f^*P|$ is base-component free and its general element is mapped generically one-to-one to $S_0$.
\end{claim}

\begin{proof}
If $P^2=0$ then $|P|$ is base-point free, whence $|f^*P|$ is base-point free as well. If $P^2=2$ then $|P|$ has two distinct base points which lie outside of $Z$ by property (C1), whence $|f^*P|$ has still only two base points.
Choose any $x \in Z$ and denote by $\mathfrak{e}_x$ the exceptional curve of $f$ over $x$. Choose the unique $P_0 \in |P|$ passing through $x$. Then by property (C3) we have that $P_0$ is smooth and irreducible and $P_0 \cap Z =\{ x \}$. 
It follows that $f^* P \sim f^*P_0 = \tilde{P_0}+\mathfrak{e}_x$,
with $ \tilde{P_0}$ irreducible satisfying $\tilde{H}.\tilde{P_0}=5$ or $7$, which is neither divisible by $2$ nor $3$. Therefore $\varphi$ maps $\tilde{P_0}$ generically one-to-one to a curve of degree five or seven on $S_0$, whence different from the line $\varphi(\mathfrak{e}_x)$. 
We conclude that $\varphi$ is generically one-to-one on $\tilde{P_0}+\mathfrak{e}_x$, whence on the general element of $|f^*P|$. 
\end{proof}

Now recall that, by \cite[Prop.0.3.4]{cd} and \cite[Thm.1]{ko} (or \cite{gr}), any anticanonical del Pezzo cubic surface in $\PP^3$ is $\mathbb Q$-factorial and it contains at least one line and one pencil of conics (which is complete). For such a pencil $|\mathfrak{D}|$ we define its {\it strict transform} to be the moving part of the pencil $\{ \varphi^*D_0 \}_{D_0 \in |\mathfrak{D}|}$. Note that it is complete. 

\begin{cor} 
\label{cor:eecc1}
Let $|\tilde{D}|$ be the strict transform of a pencil of conics on $S_0$. Then
\[ \tilde{D} \sim f^*D- \sum \alpha_i \mathfrak{e}_i, \; \alpha_i \geq 0 \]
for some nef $D \in \Pic S$ with $D^2 \geq 4$ and $\sum \alpha_i \geq 3$.
\end{cor}

\begin{proof}
We have $\tilde{D}.\tilde{H}=6$ and any member of $|\tilde{D}|$ is mapped generically $3:1$ by $\varphi$ to a conic on $S_0$. Therefore $\tilde{D} \sim f^*D- \sum \alpha_i \mathfrak{e}_i$ for some $D \in \Pic S$ such that $|D|$ is base-component free. Moreover if $\dim |D| = 1$ then the blown-up points in $S$ are base-points of $|D|$ with multiplicity $\alpha_i$. If $D^2 = 0$ this can happen only if $\alpha_i = 0$ for every $i$ and $|D|$ is a pencil on $S$ such that $(f^*D).\tilde{H}=6$, contradicting Claim \ref{cl:eecc1}.  If $D^2 = 2$ this can happen only if $\alpha_i = 0$ for every $i$ except two of them, say $\alpha_j = \alpha_k = 1$ and $|D|$ is a pencil on $S$ such that $(f^*D).\tilde{H}=8$, again contradicting Claim \ref{cl:eecc1}. Hence $D^2 \geq 4$ and the Hodge index theorem yields $D.L \geq 9$, whence $6=\tilde{D}.\tilde{H} =D.L - \sum \alpha_i$ implies $\sum \alpha_i \geq 3$.
\end{proof}

\begin{claim} 
\label{cl:eecc1'}
For any distinct $i,j,k \in \{1, \ldots, 9\}$ we have
\[ h^0(\tilde{H} -\mathfrak{e}_i -\mathfrak{e}_j-\mathfrak{e}_k) \leq 1. \]
\end{claim}

\begin{proof}
Assume, to get a contradiction, that $h^0(\tilde{H} -\mathfrak{e}_1-\mathfrak{e}_2-\mathfrak{e}_3) \geq 2$. Then, from
\[ 0 \hpil f^*K_S-\mathfrak{e}_1-\mathfrak{e}_2-\mathfrak{e}_3  \hpil \tilde{H} -\mathfrak{e}_1-\mathfrak{e}_2-\mathfrak{e}_3   
\hpil \O_{\tilde{C}}(\tilde{H} -\mathfrak{e}_1-\mathfrak{e}_2-\mathfrak{e}_3)  \hpil   0, \] 
we find $h^0(\O_{\tilde{C}}(\tilde{H} -\mathfrak{e}_1-\mathfrak{e}_2-\mathfrak{e}_3)) \geq 2$. But $\O_{\tilde{C}}(\tilde{H} -\mathfrak{e}_1-\mathfrak{e}_2-\mathfrak{e}_3) \iso \O_C(A-x_1-x_2-x_3) \iso \O_C(x_4 + \cdots +x_9)$, where the $x_i$'s are the nine points of $Z$.  But this means that $|\O_C(x_4 + \cdots +x_9)|$ is a $g^1_6$, contradicting property (C2). 
\end{proof}

\begin{claim} 
\label{cl:eecc2}
Let $\tilde{G}$ be an irreducible curve on $\tilde{S}$ different from the $\mathfrak{e}_i$'s such that $\varphi$ maps $\tilde{G}$ generically $1:1$ or $2:1$ to a line or to a point.

\noindent Then $\tilde{G} \sim f^*G$ for some effective irreducible $G \in \Pic S$.
\end{claim}

\begin{proof}
Assume, to get a contradiction, that $\tilde{G} \sim f^*G - \sum \beta_i \mathfrak{e}_i$, with at least one $\beta_i >0$. This means that $G \cap Z \neq \emptyset$, so that by property (C1) we must have $h^0(G) \geq 2$ and consequently $G.L \geq 2\phi(L) \geq 6$. By assumption we have $2 \geq \tilde{G}.\tilde{H}=G.L- \sum \beta_i$, whence $\sum \beta_i \geq 4$. Now we cannot have $\length(G \cap Z) = 1$, for then $\tilde{G} \sim f^*G- \beta \mathfrak{e}_x$ with $\beta \geq 4$ and one exceptional curve $\mathfrak{e}_x$ lying over the only intersection point $x$ between $G$ and $Z$. Hence
\[ 2p_a(\tilde{G})-2 = \tilde{G}.(\tilde{G}+K_{\tilde{S}}) = G^2- \beta^2+\beta = 4-\beta.(\beta-1) \leq -8,\]
an absurdity. Therefore $\length(G \cap Z) \geq 2$ and from property (C3), by \cite[Prop.3.1.6 and 3.1.4]{cd}, we deduce that $G^2 \geq 4$.

By assumption we have $\tilde{H} \geq \tilde{G} + \tilde{D}$, where $|\tilde{D}|$ is the strict transform of a pencil of conics on $S_0$. It follows that $L+K_S \geq G+D$. Moreover $\tilde{D} \sim f^*D- \sum \alpha_i \mathfrak{e}_i$, with $D^2 \geq 4$ and $\sum \alpha_i \geq 3$ by Corollary \ref{cor:eecc1}. The Hodge index theorem yields $G.L \geq 9$ and $D.L \geq 9$. Now both $G$ and $D$ are nef, whence $18 = L^2 \geq L.(G+D) \geq D^2 + G^2 + 2D.G$ and we get that there exists $D'  \geq D$ such that
\begin{eqnarray}
\label{eq:eecc1'''}
G^2=(D')^2=4, \; G.D'=5, \; L.G = 9 \; \mbox{and} \; L+K_S \sim D' + G. 
\end{eqnarray}
Again from $\length(G \cap Z) \geq 2$ and property (C4) it follows that $G$ is smooth and that if furthermore 
$\phi(G)=2$, then $G$ is nonhyperelliptic. Now $p_a(G)=3$, whence $\varphi$ cannot map $\tilde{G}$ generically $1:1$ to a line, and if $\varphi$ maps $\tilde{G}$ generically $2:1$ to a line, then we must have 
$\phi(G)=1$. Using \cite[Lemma2.4]{kl1} we can write $G \sim 2E_1+E_2$, with $E_i >0$, $E_i^2=0$ and $E_1.E_2=1$. Now $5=(2E_1+E_2).D'$ together with $3 \leq \phi(L) \leq E_1.L = E_1.G+ E_1.D'= 1 + E_1.D'$ and $3 \leq \phi(L) \leq E_2.L = E_2.G + E_2.D'= 2 + E_2.D'$ imply $E_1.D'=2$ and $E_2.D'=1$, and one easily sees that this implies $D' \eqv E_1+2E_2$, so that $L \eqv 3(E_1+E_2)$, contrary to our assumptions.

We are left with the case of $\varphi$ contracting $\tilde{G}$ to a point, that is $\tilde{H}.\tilde{G}=0$. This implies $\sum \beta_i=9$ and since $G$ is smooth, we must have
\begin{equation}
\label{eq:eecc1'}
\tilde{G} = f^* G - \sum_{i=1}^9 \mathfrak{e}_i.
\end{equation}
Moreover, as $0 \leq \tilde{D^2} = D^2- \sum \alpha_i^2 = 4- \sum \alpha_i^2$ and $\sum \alpha_i \geq 3$, we must have 
$\alpha_i=1$ for exactly three or four distinct $i$'s and $\alpha_i=0$ for the rest. Possibly after rearranging indices we can therefore write
\begin{equation}
\label{eq:eecc1''}
\tilde{D} = f^* D - \mathfrak{e}_1-\mathfrak{e}_2-\mathfrak{e}_3 -\varepsilon \mathfrak{e}_4, \; \mbox{with} \; 
\varepsilon=0 \; \mbox{or} \; 1. 
\end{equation}
Now for some $\tilde{G}_1 \geq 0$ we have $\tilde{H} \sim \tilde{G}_1 + \tilde{D} + \tilde{G}$
and combining with \eqref{eq:eecc1'''}-\eqref{eq:eecc1''} we get
\begin{eqnarray*}
\tilde{G}_1 \sim & f^*(L+K_S - G - D) + \mathfrak{e}_1 + \mathfrak{e}_2 + \mathfrak{e}_3 +\varepsilon \mathfrak{e}_4,  
\end{eqnarray*}
whence $\mathfrak{e}_i.\tilde{G}_1 = -1$ for $1 \leq i \leq 3$. Therefore $\tilde{G}_1 \geq \mathfrak{e}_1 + \mathfrak{e}_2 + \mathfrak{e}_3$, whence
\[ h^0(\tilde{H}-\mathfrak{e}_1-\mathfrak{e}_2-\mathfrak{e}_3)=h^0(\tilde{G}_1-\mathfrak{e}_1-\mathfrak{e}_2-\mathfrak{e}_3 + \tilde{D} + \tilde{G}) \geq h^0(\tilde{D}) =2, \]
contradicting Claim \ref{cl:eecc1'}.
\end{proof}

We will now denote by $\mathfrak{L_j}$ the line $\varphi(\mathfrak{e}_j)$, for $j=1,\ldots, 9$ (note that these lines may coincide), and by $\mathfrak{D}_j$ the pencil of conics on $S_0$ given by the hyperplanes through $\mathfrak{L_j}$ (in other words $\O_{S_0}(1) \sim \mathfrak{L_j} + \mathfrak{D}_j$). We denote the strict transform of this pencil by $|\tilde{D_j}|$. In particular $\tilde{H}.\tilde{D_j}=6$ and by Corollary \ref{cor:eecc1} we have
\begin{equation}
\label{eq:eecc2}
\tilde{D_j} \sim f^*D_j- \sum_{i=1}^9 \alpha_{ji} \mathfrak{e}_i, \; \mbox{with} \;  D_j^2 \geq 4 \; 
\mbox{and} \;  \sum \alpha_{ji} \geq 3.  
\end{equation}
We have $\tilde{H} \sim \varphi^*(\mathfrak{L_j} + \mathfrak{D}_j)$, which yields,
for each $j=1, \ldots, 9$,
\begin{equation}
\label{eq:eecc3}
\tilde{H} \sim \tilde{\Delta}_{0j} + \tilde{\Delta}_{1j} + \tilde{D_j}, 
\end{equation}
where $\tilde{\Delta}_{1j}>0$ such that none of its components are contracted by $\varphi$ and $\varphi(\tilde{\Delta}_{1j})=\mathfrak{L_j}$, and $\tilde{\Delta}_{0j} \geq 0$ is contracted by $\varphi$, that is $\tilde{H}.\tilde{\Delta}_{0j}=0$.

Clearly $\mathfrak{e}_j \sub \tilde{\Delta}_{1j}$ by construction, so that three cases may occur:
\begin{equation}
\label{eq:eecc4}
\tilde{\Delta}_{1j} = \mathfrak{e}_j + \mathfrak{e}_a +\mathfrak{e}_b, \; \mbox{for some} \; a,b \in \{1, \ldots, 9 \}, a \neq b;
\end{equation}
or
\begin{equation}
\label{eq:eecc5}
\tilde{\Delta}_{1j} = \mathfrak{e}_j + \mathfrak{e}_a +\tilde{\Gamma}_j, \; \mbox{for some} \; a \in \{1, \ldots, 9 \},
\end{equation}
with $\tilde{\Gamma}_j$ irreducible being mapped generically $1:1$ to $\mathfrak{L_j}$ by $\varphi$, $\tilde{\Gamma}_j \neq \mathfrak{e}_i$ for all $i$; or
\begin{equation}
\label{eq:eecc6}
\tilde{\Delta}_{1j} = \mathfrak{e}_j +\tilde{\Gamma}_j,
\end{equation}
with $\tilde{\Gamma}_j$ either irreducible and being mapped generically $2:1$ to $\mathfrak{L_j}$ by $\varphi$ or consisting of two irreducible components $\neq \mathfrak{e}_i$ for all $i$.

\noindent By Claim \ref{cl:eecc2} we have, in all cases, that
\begin{equation}
\label{eq:eecc7}
\tilde{\Gamma}_j = f^* \Gamma_j \; \mbox{and} \; \tilde{\Delta}_{0j} = f^*\Delta_{0j}
\end{equation}
From $\tilde{H} \sim f^*(L+K_S) - \sum_{i=1}^9 \mathfrak{e}_i$, \eqref{eq:eecc3}, \eqref{eq:eecc7} and \eqref{eq:eecc2} we get, for each $p \in \{1, \ldots, 9\}$,
\[ 1 = \tilde{H}.\mathfrak{e}_p = (\tilde{\Delta}_{0j} + \tilde{\Delta}_{1j} + \tilde{D_j}).\mathfrak{e}_p = (f^*\Delta_{0j}).\mathfrak{e}_p + \tilde{\Delta}_{1j}.\mathfrak{e}_p + ( f^*D_j- \sum_{i=1}^9 \alpha_{ji} \mathfrak{e}_i).\mathfrak{e}_p = \tilde{\Delta}_{1j}.\mathfrak{e}_p + \alpha_{jp}. \]
Using \eqref{eq:eecc4}-\eqref{eq:eecc7} we deduce that
\[\alpha_{jp} =  \begin{cases} 1 & {\rm if} \ p \not\in \{j, a, b \} \\ 2 & {\rm if} \ p \in \{j, a, b \} \end{cases} \mbox{\ in case} \ \eqref{eq:eecc4} \ \mbox{and} \ \alpha_{jp} =  \begin{cases} 1 & {\rm if} \ p \not\in \{j, a \} \\ 2 & {\rm if} \ p \in \{j, a \} \end{cases} \mbox{\ in case} \ \eqref{eq:eecc5}.\]  
Moreover, we have
\begin{equation} 
\label{eq:eecc8}
L+K_S \sim D_j+ \Gamma_j+ \Delta_{0j} \; 
\mbox{(with $\Gamma_j=0$ in case \eqref{eq:eecc4})}.
\end{equation}
If we are in case \eqref{eq:eecc4}, from \eqref{eq:eecc2} and \eqref{eq:eecc8} we have
\[ 0 \leq \tilde{D_j}^2 = D_j^2 - \sum_{i=1}^9 \alpha_{ji}^2  \leq 18+ \Delta_{0j}^2 -18 = \Delta_{0j}^2, \]
which implies $\Delta_{0j}=0$. Reordering indices we can from \eqref{eq:eecc3} assume that 
$\tilde{H} \sim \mathfrak{e}_1+\mathfrak{e}_2+\mathfrak{e}_3 + \tilde{D_j}$,
whence $h^0(\tilde{H} -\mathfrak{e}_1-\mathfrak{e}_2-\mathfrak{e}_3)=2$, contradicting Claim \ref{cl:eecc1'}. 

If we are in case \eqref{eq:eecc5}, then $\Gamma_j.L=1, \Gamma_j^2=-2$ and we claim that $D_j ^2 \leq (L- \Gamma_j)^2$. The latter being obvious if $\Delta_{0j}=0$ (using \eqref{eq:eecc8}), we assume $\Delta_{0j}>0$. By \eqref{eq:eecc7} we have $0 = \tilde{\Delta}_{0j}.\tilde{H} = \tilde{\Delta}_{0j}.(f^*L- \sum \mathfrak{e}_i) = f^*(\Delta_{0j}).(f^*L- \sum \mathfrak{e}_i) = \Delta_{0j}.L$. Now write $\tilde{\Delta}_{0j} = \sum_{q} \tilde{G}_{qj}$ with $\tilde{G}_{qj} > 0, \Supp(\tilde{G}_{qj})$ connected and $\tilde{G}_{qj}.\tilde{G}_{pj} = 0$ for $q \neq p$. Then $\tilde{H}.\tilde{G}_{qj} = 0$ for every $q$ and, as in \eqref{eq:eecc7}, $\tilde{G}_{qj} = f^{\ast} G_{qj}$. Therefore $L.G_{qj} = 0$ for every $q$ and the Hodge index theorem and \eqref{eq:eecc7} imply that $\tilde{G}_{qj}^2 = G_{qj}^2 \leq -2$. Moreover, as $\varphi$ maps $\tilde{\Gamma}_j$ to a line and $\tilde{G}_{qj}$ to a point, we have that $\tilde{\Gamma}_j.\tilde{G}_{qj} \leq 1$, whence 
\[ 2\Gamma_j.\Delta_{0j} + \Delta_{0j}^2 = \tilde{\Delta}_{0j}.(2 \tilde{\Gamma}_j + \tilde{\Delta}_{0j}) = \sum_{q} \tilde{G}_{qj}.(2 \tilde{\Gamma}_j + \sum_{q} \tilde{G}_{qj}) = \sum_{q} (2 \tilde{G}_{qj}.\tilde{\Gamma}_j + \tilde{G}_{qj}^2) \leq 0. \]
Therefore
\[ D_j ^2= (L- \Gamma_j - \Delta_{0j})^2 = (L- \Gamma_j)^2 -2 (L-\Gamma_j).\Delta_{0j} + \Delta_{0j}^2 =
(L- \Gamma_j)^2 + 2\Gamma_j.\Delta_{0j} + \Delta_{0j}^2 \leq (L- \Gamma_j)^2. \]
But this yields the contradiction $\tilde{D_j}^2=D_j^2 - \sum_{i=1}^9 \alpha_{ji}^2 \leq (L- \Gamma_j)^2 -15 = - 1$.
Therefore we must be in case \eqref{eq:eecc6} for all $j$. In particular, using \eqref{eq:eecc7},
\begin{equation}
\label{eq:eecc9}
\varphi^*(\mathfrak{L_j}) \leq \mathfrak{e}_j+\tilde{\Gamma}_j+ \tilde{\Delta}_{0j} \sim
                               \mathfrak{e}_j+f^*(\Gamma_j+\Delta_{0j}), \; \mbox{for all} \; j.
\end{equation}
This implies that all the $\mathfrak{L_j}$ are distinct lines. As $S_0$ cannot contain more than $6$ mutually disjoint lines (since otherwise its minimal desingularization $\tilde{S}_0$ would contain at least $7$ disjoint $(-1)$-curves, which is impossible, since it is obtained by blowing up $\PP^2$ in $6$ points), we must have $\mathfrak{L_i} \cap \mathfrak{L_j} \neq \emptyset$ for some $i \neq j$. But then $|\mathfrak{L_i} + \mathfrak{L_j}|$ is a pencil of conics on $S_0$, whence by \eqref{eq:eecc9} and \eqref{eq:eecc7} we have
\[ \dim |\mathfrak{e}_i+\mathfrak{e}_j+f^*(\Gamma_i+\Delta_{0i}+\Gamma_j+\Delta_{0j})| \geq 1. \]
But since $\mathfrak{e}_i.(\mathfrak{e}_i+\mathfrak{e}_j+f^*(\Gamma_i+\Delta_{0i}+\Gamma_j+\Delta_{0j}))=-1$ we see that $\mathfrak{e}_i$, and similarly $\mathfrak{e}_j$, are both fixed in $|\mathfrak{e}_i+\mathfrak{e}_j+f^*(\Gamma_i+\Delta_{0i}+\Gamma_j+\Delta_{0j})|$. Therefore
\[ \dim |f^*(\Gamma_i+\Delta_{0i}+\Gamma_j+\Delta_{0j})| \geq 1, \]
but since $f^*(\Gamma_i+\Delta_{0i}+\Gamma_j+\Delta_{0j}).\tilde{C}= (\Gamma_i+\Delta_{0i}+\Gamma_j+\Delta_{0j}).L=4$ we obtain a $g^1_4$ on $C$, a contradiction.

This concludes the proof of the case $r = 3$.

\section{Curves of Clifford dimensions from $4$ to $9$}
\label{sec:r=4-9}

In this section we will show that there is no exceptional curve $C$ of Clifford dimension $r$, with $4 \leq r \leq 9$, on an Enriques surface $S$.  Set $L = \O_S(C)$ and let $A$ be a line bundle on $C$ computing the Clifford dimension. We start with the following result

\begin{lemma} 
\label{useful2}
Assume $C$ is an exceptional curve of Clifford dimension $r$ with $4 \leq r \leq 9$ on an Enriques surface $S$. Then $r \leq 6$ and
\[ C^2=8r-6, \; \phi(C)=r, \; \Cliff C = 2r-3, \; \gon C =2r. \]
Moreover there is a unique line bundle $A$ computing the Clifford dimension and it satisfies $\omega_C \sim 2A$ and $\deg A=4r-3$.
\end{lemma}

\begin{proof}
Since $r \leq 9$ then from \cite[End of \S 5, Thm. 3.6 and Thm. 3.7]{elms} it follows that $A$ is unique, $\omega_C \sim 2A$, $h^0(A)=h^1(A) =r+1$, $\Cliff C =2r-3$ and $g(C)=4r-2$, whence $C^2 =8r-6$ and $\gon C = 2r$. From \eqref{eq:E2a} we have $8 \phi(C)  \geq  8r = C^2 + 6 \geq \phi(C)^2+6$, which yields $\phi(C) \leq 7$, whence $r \leq 7$ by \eqref{eq:E2a}. By \cite[Proposition1.4]{kl1}, we get $r=\phi(C) \leq 6$.
\end{proof}

\begin{lemma} 
\label{lemma:exc2}
Let $N$ be a nef line bundle on an Enriques surface $S$ with $(N^2, \phi(N))=(42,6)$ or $(34,5)$. 
Then there is an effective divisor $B$ on $S$ satisfying $B^2=2$ and $B.N=2\phi(N)$.
\end{lemma}

\begin{proof}
Choose an $E>0$ such that $E^2=0$ and $E.N=\phi(N)$. Set $N_1=N-E$, which is effective by \cite[Lemma2.4]{kl1}, and choose an $E_1 >0$ such that $E_1^2=0$ and $E_1.N_1=\phi(N_1)$. 

We first treat the case $(N^2, \phi(N))=(42,6)$. Then $N_1^2=30$, whence $E_1.N_1=\phi(N_1) \leq 5$.

If $\phi(N_1) \leq 4$, then $(N_1-3E_1)^2 \geq 6$ and $N_1-3E_1 > 0$ by \cite[Lemma2.4]{kl1}. Now $6=\phi(N) \leq E_1.N =E_1.N_1+E_1.E \leq 4 + E_1.E$ implies $E_1.E \geq 2$, whence $6=E.N=E.N_1 = 3E.E_1 + E.(N_1-3E_1) \geq 7$, a contradiction.
Therefore $\phi(N_1)=5$, so that $(N_1-3E_1)^2 =0$. If $E.E_1 \geq 2$ then $6=E.N=E.N_1 = 3E.E_1 + E.(N_1-3E_1)$
implies $N_1-3E_1 \eqv kE$ for some $k \geq 1$ by \cite[Lemma2.1]{klvan}. Now $5=E_1.(N_1-3E_1)=5 kE.E_1 \geq 10$ gives a contradiction. Hence $E.E_1=1$, so that $E_1.N=E.N=6$ and $E+E_1$ is the desired divisor.
The case $(N^2, \phi(N))=(34,5)$ follows in the same manner.
\end{proof}

We now choose $B$ as in Lemma \ref{lemma:exc2} in the cases $r=5$ and $6$. If there is a divisor satisfying $B^2=2$ and $B.L=2r=8$ in the case $r=4$, we pick such a divisor in that case as well, if not we choose an elliptic pencil $|2E|$ with $E.L=4$ and set $B=2E$ in this case. To summarize, we fix an effective  divisor $B$ from now on with the following properties:
\begin{eqnarray} 
\label{eq:exc3} B.L & = & 2\phi(L)=2r \; \mbox{and either} \\
\nonumber & {\rm (i)} & B^2=2 \; \mbox{or} \\
\nonumber   & {\rm (ii)} & (L^2,r)=(26,4), \; B \; \mbox{is an elliptic pencil and there is no} \\
\nonumber & & B_0 >0 \; \mbox{satisfying} \; B_0^2=2 \; \mbox{and} \; B_0.L=8.
\end{eqnarray}

By Lemma \ref{useful2}, we must have $h^0(B) = h^0(\O_C(B))=2$ and $|B_{|C}|$ is a base-point free pencil, whence
\begin{equation} 
\label{eq:exch1=0}
h^1(B)=0 \; \mbox{if} \; B^2=2.
\end{equation}

By \cite[Lemma3.1]{elms} there is an effective divisor $Z$ on $C$ of degree $2r-3$ such that $\O_C(B) \sim A - Z$. Since $\omega_C \sim 2A$, by Lemma \ref{useful2}, setting $D=L+K_S-B$, this yields
\[ A \sim \O_C(D) - Z. \]
From the cohomology of
\begin{equation*} 
0 \hpil K_S-B \hpil \I_Z(D) \hpil A \hpil 0
\end{equation*}
we find
\begin{equation*} 
H^0(\I_Z(D))  \sub H^0(A), 
\end{equation*}
with strict inclusion implying $\codim_{H^0(A)} H^0(\I_Z(D)) =1$ and $B^2=0$ by \eqref{eq:exch1=0}.

To simplify our treatment we will refer to the two different cases as (I) and (II), that is:
\begin{eqnarray} 
\label{eq:exc6} \mbox{Case (I):} & H^0(\I_Z(D)) = H^0(A), \\
 \nonumber      \mbox{Case (II):} & 
            \codim_{H^0(A)} H^0(\I_Z(D)) =1, \; \mbox{in which case} \; 
B^2=0.
\end{eqnarray}

Now let $F$ be the base-component of $|\I_Z(D)|, M = D - F$ and $Z_0 = Z \cap F$ be the scheme-theoretic intersection. Then it follows that there is an effective decomposition $Z=Z_M +Z_0$, as divisors on $C$, such that
\begin{equation} 
\label{eq:exc7'} 
|\I_Z(D)| = |\I_{Z_M} \* M| + F, \ M \; \mbox{is base-component free} 
\end{equation} 
and
\begin{equation} 
\label{eq:exc7} 
\O_C(M)-Z_M \sim A - ( F \cap C - Z_0).
\end{equation}  
From  
\begin{equation} 
\label{eq:exc8}
0 \hpil M-L \hpil \I_{Z_M} \* M \hpil \O_C(M)-Z_M \hpil 0
\end{equation}
we get
\begin{equation} 
\label{eq:exc9}
h^0(\O_C(M)-Z_M) \geq h^0(\I_{Z_M} \* M)=h^0(\I_{Z}(D)) = 
 \left\{ \begin{array}{ll}
          r+1 & \mbox{in Case (I)}, \\
          r   & \mbox{in Case (II).}
        \end{array}
\right. 
\end{equation}
Moreover, $h^1(\O_C(M)-Z_M) \geq h^2(M-L)=h^0(L-M+K_S) \geq h^0(B)=2$,
whence $\O_C(M)-Z_M$ contributes to the Clifford index of $C$.

If $h^0(\O_C(M)-Z_M) \geq r+1$ we therefore have, by \eqref{eq:exc7} and Lemma \ref{useful2},
\begin{eqnarray*}
  \Cliff C  & \leq & \Cliff \O_C(M)(-Z_M) \leq \deg(\O_C(M)-Z_M)-2r  = \\
                     & = & \deg A -2r - \deg (F \cap C - Z_0) = \Cliff C - 
 \deg (F \cap C - Z_0),
\end{eqnarray*}
whence $Z_0 = F \cap C$ and $\O_C(M)-Z_M  \sim A$. 

If $h^0(\O_C(M)-Z_M) =r$, then, since $r$ is the Clifford dimension of $C$, we have
\[ \Cliff C  <  \Cliff \O_C(M)(-Z_M) = \Cliff C +2 - \deg (F \cap C - Z_0), \]
whence $\deg (F \cap C - Z_0) \leq 1$. However, as $h^0(\O_C(M)-Z_M) < h^0(A)$,
we must have $\deg (F \cap C - Z_0) =1$.
We therefore conclude that
\begin{eqnarray} 
\label{eq:exc10}   \mbox{Either} &  & \O_C(M)-Z_M  \sim A, \; \mbox{or} \\
\nonumber & & \O_C(M)-Z_M \sim A-x, \; \mbox{for some} \; x \in C, \; \mbox{in which case}  \\
\nonumber   & & h^0(\O_C(M)-Z_M) = h^0(\I_{Z_M} \* M) = r \;  \mbox{and we are in Case (II).}
 \end{eqnarray}
Now take cohomology of \eqref{eq:exc8} and set
\[ V = \Im \Big\{ H^0(\I_{Z_M} \* M) \hpil H^0(\O_C(M)-Z_M) \Big \}. \]
Clearly
\begin{equation} 
\label{eq:exc11}
\dim V= h^0(\I_{Z_M} \* M)=  \left\{ \begin{array}{ll}
          r+1 & \mbox{in Case (I)}, \\
          r   & \mbox{in Case (II).}
        \end{array}
\right. 
\end{equation}
The evaluation map $ev_{(\I_{Z_M} \* M)} : H^0(\I_{Z_M} \* M) \* \O_S \khpil \I_{Z_M} \* M$ is surjective off a finite set since $|\I_{Z_M} \* M|$ is without fixed components and its kernel is a vector bundle whose dual we denote by $\F$, while its cokernel is a torsion sheaf with finite support that will be
denoted by $\tau_M$. Similarly, the evaluation map $ev_V: V \* \O_S \hpil \O_C(M)-Z_M$ is surjective off a finite set and its kernel is a vector bundle whose dual we denote by $\E_0$, while its cokernel is a torsion sheaf with finite support that will be
denoted by $\tau_V$. Set $l_V =\length(\tau_V)$ and $l_M = \length(\tau_M)$. Note that 
\begin{equation} 
\label{eq:nuova}
\mbox{if} \ l_V = 0 \ \mbox{then} \ h^0(\E_0(K_S)) = h^2(\E_0^*) = h^1(\O_C(M)(-Z_M)).
\end{equation}
Taking evaluation maps in \eqref{eq:exc8}, applying the snake lemma and dualizing yields 
\begin{equation} 
\label{eq:exc14} 
0 \hpil \O_S(L-M) \hpil \E_0 \hpil \F   \hpil \tau \hpil 0, 
\end{equation}  
where $\tau$ is a torsion sheaf of finite support of length $l:=l_M-l_V \geq 0$.

Now note that we have
$c_1(\E_0) \sim L \; \mbox{and} \; c_1(\F) \sim M$
and, from \eqref{eq:exc11} and \eqref{eq:exc14}, 
\begin{equation} 
\label{eq:exc16}
 \rk \E_0= \dim V= \left\{ \begin{array}{ll}
          r+1 & \mbox{in Case (I),}  \\
          r   & \mbox{in Case (II),}
        \end{array}
\right. \; \mbox{and} \; \rk \F =\rk \E_0-1
\end{equation} 
(note that it follows that $\E_0 \iso \E(C,A)$ in case (I))  and
\begin{equation}
\label{eq:exc17}
c_2(\E_0) = \deg (\O_C(M)-Z_M) -l_V.
\end{equation} 
Moreover we have
\begin{equation} 
\label{eq:extra}
h^0(\E_0^*)=h^0(\F^*)=0 \; \mbox{and} \; h^1(\E_0^*) \leq h^0(\O_C(M)(-Z_M))-h^0(\I_{Z_M} \* M).
\end{equation} 
Taking $c_2$ in \eqref{eq:exc14} and combining with \eqref{eq:exc10} and \eqref{eq:exc17} we obtain
\begin{equation}
\label{eq:exc18}
4r-3 \geq \deg (\O_C(M)-Z_M)= M.(L-M) + c_2(\F) +l_M.
\end{equation} 
We also note that by dualizing the evaluation sequence of $\I_{Z_M} \* M$, that is the middle column of the commutative diagram above, we see that,
\begin{equation} 
\label{eq:exc19}
\F \; \mbox{is globally generated off a finite set and, if} \; l_M=0 \; \mbox{and} \; M^2 >0, \; \mbox{then} \; \F \; 
\mbox{is good.}
\end{equation} 
(Indeed, the latter is an immediate consequence of Lemma \ref{lemma:good}, since $Z_M \subset C$ is curvilinear.)

Now if $M^2=0$, then $M \sim mP_0$ for an elliptic pencil $|P_0|$ and $m \geq 1$ and $c_2(\F) \geq 2\rk \F-2m \geq 2(r-1-m)$ by \cite[Prop.3.2(b)]{kn2} and \eqref{eq:exc16}, whence inserting into  \eqref{eq:exc18} we get
$4r-3 \geq mP_0.L + 2(r-1-m) \geq 2(mr+r-1-m)$.
Therefore $m=1$ and $h^0(M)=2$. But from \eqref{eq:exc9} we must have $h^0(M) \geq h^0(\I_{Z_M} \* M) \geq r \geq 4$, a contradiction. Therefore $M^2 >0$, and again from \eqref{eq:exc9} together with the fact that $h^1(M)=0$ as $M$ is base-component free, we must have 
\begin{eqnarray} 
\label{eq:exc22}  M^2 & \geq & 2r-2 \; \mbox{with equality only in case (II) with} \\
\nonumber & & h^0(M) = h^0(\I_{Z_M} \* M)=r. 
\end{eqnarray}

Now note that $(L-M).L \geq (L-D).L =B.L=2r$, and since $L^2=8r-6$, also
\begin{equation} 
\label{eq:exc30}
L.M \leq 6r-6.
\end{equation} 

Next we claim that
\begin{equation} 
\label{eq:exc31}
\phi(M) \geq 2.
\end{equation} 
Indeed, if $\phi(M)=1$, one easily sees, by \cite[Lemma2.4]{kl1} that one can write $M \sim kE_1+E_2$ for $E_i >0$, $E_i^2=0$ and $E_1.E_2=1$ and $k:= \frac{1}{2}M^2 \geq r-1$ by \eqref{eq:exc22}. Therefore $M.L \geq (k+1)\phi(L) \geq r^2$ and combining with \eqref{eq:exc30} we find that $r=4$, and again by \eqref{eq:exc22} we find $M^2=6$ and $h^0(M)=h^0(\I_{Z_M} \* M)=4$. Therefore $Z_M$ is contained in the base locus of $|M|$, so that $\deg Z_M \leq 2$, whence $M.L=M.C \leq \deg A+2=15$ by \eqref{eq:exc10}, a contradiction.

Therefore \eqref{eq:exc31} holds. We now claim that we have
\begin{equation} 
\label{eq:exc23}
M.(L-M) \geq 2r-2 \; \mbox{and} \; M.(L-M) \geq 2r \; \mbox{if} \; h^1(M-L) \geq 1 \; \mbox{or} \; (L-M)^2 \leq 0.
\end{equation} 

Indeed, if $(L-M)^2 \leq 0$ this simply follows since $(L-M).L \geq (L-D).L=B.L=2r$. For the rest we note that $L+K_S \sim M + (L-M+K_S)$ is an effective decomposition with $h^0(L-M+K_S) \geq h^0(B) = 2$, so that
$M.(L-M) \geq \Cliff C +2h^1(L-M+K_S) = 2r-3 +2h^1(L-M+K_S)$ by Lemma \ref{cliffhanger}. 

If $M.(L-M) =2r-3$, then $h^1(L-M+K_S)=0$ and both $\O_C(M)$ and $\O_C(L-M+K_S)$ compute the Clifford index of $C$ by Lemma \ref{cliffhanger} and therefore $h^0(\O_C(M))=h^0(M) \geq r+1$ and $h^0(\O_C(L-M+K_S))=h^0(L-M+K_S) \geq r+1$,
whence $M^2 \geq 2r$ and $(L-M)^2 \geq 2r$, but this is easily seen to contradict the Hodge index theorem.

If $h^1(L-M+K_S) >0$ and $M.(L-M) \leq 2r-1$, then we must have
$h^1(L-M+K_S) =1$ and $M.(L-M) =2r-1$, so that $h^0(\O_C(M)) \leq h^0(M)+1$.
Moreover $h^0(\O_C(L-M+K_S))=h^0(L-M+K_S)$ as $h^1(M)=0$.
Again by Lemma \ref{cliffhanger}, both $\O_C(M)$ and $\O_C(L-M+K_S)$ compute the Clifford index of $C$, 
whence $M^2 \geq 2r-2$ and $(L-M)^2 \geq 2r-2$. Now
\[ 8r-6 = L^2 = M^2 + (L-M)^2 +2M.(L-M) \geq 8r-6 \]
implies that $M^2 = 2r-2$, so that we must be in Case (II) by \eqref{eq:exc22} and $r = \phi(L) = 4$. Therefore $M^2 = 6$ and $M.L= 13$. By \eqref{eq:exc31}, $\phi(M) = 2$ and again by \cite[Lemma2.4]{kl1}, we can write $M \sim F_1+F_2+F_3$ for $F_i >0$, $F_i^2=0$ and $F_i.F_j=1$ if $i \neq j$. But now $M.L=13$ implies that at least two of the $F_i$'s satisfy $F_i.L=4$, contradicting \eqref{eq:exc3}.
This settles \eqref{eq:exc23}.

\subsection{Case (I)}
Recall that $\O_C(M)-Z_M = A$ by \eqref{eq:exc10} whence $\deg (\O_C(M)-Z_M)=4r-3$ by Lemma \ref{useful2}.
By \eqref{eq:exc16} we have $\rk \F =r$. Setting $c(\F) = c_2(\F)-2(\rk \F-1)=c_2(\F)-2(r-1) \geq 0$ by Proposition \ref{prop:bassoc}, we may rewrite \eqref{eq:exc18} as 
\begin{equation*}
 M.(L-M) + l_M + c(\F) =2r-1
\end{equation*}
Comparing with \eqref{eq:exc23} we see that $(L-M)^2 >0$ and that the only possibilities are
\[ ( M.(L-M), l_M, c(\F)) = (2r-2,1,0), (2r-2,0,1), (2r-1,0,0). \]

In the two cases with $c(\F)=0$ we must have $M^2 \leq 8$ by Proposition \ref{prop:bassoc}(i) and \eqref{eq:exc31}.
Since $M^2 \geq 2r \geq 8$ by \eqref{eq:exc22}, we must have $r=4$ and $M^2=8$, whence  $M \eqv 2M_0$ with $M_0^2=2$ by Proposition \ref{prop:bassoc}(i). Therefore $M.(L-M)$ is even, so that we must be in the case with $M.(L-M)=6$ and $l_M=1$. But $L^2=26$ implies $(L-M)^2=6$ and the Hodge index theorem yields the contradiction $48= M^2 (L-M)^2 \leq 36$.

Therefore, the only case remaining is the one with $M.(L-M)=2r-2$, $l_M=0$ and $c(\F)=1$. By \eqref{eq:exc19} $\F$ is good. It is also clear that we must have $l=l_V=0$, whence twisting \eqref{eq:exc14} by $K_S$ and using \eqref{eq:exc10}, \eqref{eq:nuova} and that $h^1(\E_0(K_S))=0$ by \eqref{eq:extra}, we find
\[ h^0(\F(K_S))=h^0(\E_0(K_S)) - \chi(L-M) = h^1(A)- \chi(L-M) = r- \frac{1}{2}(L-M)^2. \]  
Therefore, if $h^0(\F(K_S)) \leq 1$, we have $(L-M)^2 \geq 2r-2$, whence the contradiction
\[ 4r(r-1) \leq M^2 (L-M)^2 \leq \Big( M.(L-M) \Big)^2 = 4(r-1)^2 \]
by \eqref{eq:exc22} and the Hodge index theorem.
Hence $h^0(\F(K_S)) \geq 2$ and we are in one of the four cases in Proposition \ref{prop:bassoc}(ii). Recall that $r \leq 6$ by Lemma \ref{useful2}.

If we are in case (ii-a) then $M.L \geq 5\phi(L)=5r$ and \eqref{eq:exc30} yields $r=6$ and $M.L=30$.
But $M.L=M^2+M.(L-M)=22$, a contradiction.

If we are in one of the cases (ii-b) or (ii-c) then $M^2=10$, whence $(L-M)^2=4(r-3)$, contradicting and the Hodge index theorem, as $r \leq 6$.

If we are in case (ii-d) then $M^2 \leq 8$. Comparing with \eqref{eq:exc22} we get $r=4$ and $M^2=8$, so that $M.(L-M)=6$. From $L^2=26=M^2 + (L-M)^2 +2M.(L-M) = 20 + (L-M)^2$ we find $(L-M)^2=6$ and the Hodge index theorem yields the same contradiction as above.

\subsection{Case (II)}

By \eqref{eq:exc6} and \eqref{eq:exc3} we have $r=4$ and by \eqref{eq:exc16} we have $\rk \F =3$. Setting $c(\F) = c_2(\F)-2(\rk \F-1)=c_2(\F)-4\geq 0$ as in Proposition \ref{prop:bassoc}, we rewrite \eqref{eq:exc18} as 
\begin{equation} 
\label{eq:exc25}
M.(L-M) + l_M + c(\F) = \deg (\O_C(M)(-Z_M))-4.
\end{equation} 

We now divide the treatment into the two cases occurring in \eqref{eq:exc10}.

\subsubsection{The case $\O_C(M)(-Z_M) \sim A$} 

In this case \eqref{eq:exc25} reads 
\begin{equation} 
\label{eq:exc26}
M.(L-M) + l_M + c(\F) = 9.
\end{equation} 
Moreover, from \eqref{eq:exc6} and \eqref{eq:exc8} we must have $h^1(M-L) \geq 1$. Therefore \eqref{eq:exc23} yields $M.(L-M) \geq 8$. Comparing with \eqref{eq:exc26}  we see that the only options are
\[ ( M.(L-M), l_M, c(\F)) = (8,1,0), (8,0,1), (9,0,0). \]
In the two cases with $c(\F)=0$ we must have $M^2 \leq 8$ by Proposition \ref{prop:bassoc}(i) and \eqref{eq:exc31}. Since $M^2 \geq 6$ by \eqref{eq:exc22}, we must have $M^2=6$ or $8$. If $M^2=8$ then  $M \eqv 2M_0$ with $M_0^2=2$ by Proposition \ref{prop:bassoc}(i). Therefore $M.(L-M)$ is even, so that we must be in the case with $M.(L-M)=8$ and $l_M=1$. Now $L^2=26=M^2 + (L-M)^2 +2M.(L-M) = 24 + (L-M)^2$ implies
$(L-M)^2=2$. Since $h^1(L-M+K_S) \neq 0$ there is a $\Delta >0$ with $\Delta^2 =-2$ and $\Delta.(L-M) \leq -2$, by \cite[Cor.2.5]{klvan}. Obviously $\Delta.M \geq 2$, whence $(M+\Delta)^2 \geq 10$, $(L-M-\Delta)^2 \geq 4$ and $(M+\Delta).(L-M-\Delta) \leq 6$, which is easily seen to contradict the Hodge index theorem. If $M^2=6$ we have $h^0(M)=h^0(\I_{Z_M} \* M)$ by \eqref{eq:exc22}, whence $Z_M=\emptyset$ since $M$ is base-point free by \eqref{eq:exc31}. It follows that $\O_C(M) \sim A$, so that $M.L= \deg A=13$, whence $M.(L-M)=7$, a contradiction. 

Therefore, the only case remaining is the one with $M.(L-M)=8$, $l_M=0$ and $c(\F)=1$. By \eqref{eq:exc19} $\F$ is good. Also we must have $l=l_V=0$, whence twisting \eqref{eq:exc14} by $K_S$ and using \eqref{eq:exc10}, \eqref{eq:nuova}, $h^1(\F(K_S))=0$ by Proposition \ref{prop:bassoc} and that $h^1(\E_0(K_S)) \leq 1$ by \eqref{eq:extra} we find
\[ h^0(\F(K_S)) \geq h^0(\E_0(K_S)) - \chi(L-M)-1 = h^1(A)- \chi(L-M)-1 = 3- \frac{1}{2}(L-M)^2. \]  
Therefore, if $h^0(\F(K_S)) \leq 1$, we have $(L-M)^2 \geq 4$. Hence by \eqref{eq:exc22} and $L^2=26$ we get $M^2=6$ and, as above, $Z_M=\emptyset$ since $M$ is base-point free and $\O_C(M) \sim A$, so that $M.L= \deg A=13$, whence $M.(L-M)=7$, a contradiction. 

Hence $h^0(\F(K_S)) \geq 2$ and we are in one of the four cases in Proposition \ref{prop:bassoc}(ii).

If we are in case (ii-a) then $M.L \geq 5\phi(L)=20$ which contradicts \eqref{eq:exc30}.

We cannot be in case (ii-b), as $\rk \F=3$.

If we are in case (ii-c) then $M^2=10$ and $\phi(M)=2$. Picking any $E_1>0$ with $E_1^2=0$ and 
$E_1.M=2$ with $E_1$ nef, one easily sees that one can write $M \sim 2E_1+E_2+E_3$, with 
$E_i>0$, $E_i^2=0$ and $E_i.E_j=1$ for $i \neq j$. Since $M.L=18$ and $E_i.L \geq 5$ for at least two of the $E_i$'s, by \eqref{eq:exc3}, we must have $E_1.L=4$ and $E_2.L=E_3.L=5$. Therefore $E_i.(L-M)=2$ for all $i=1,2,3$. Now $(L-M)^2=0$ and we have that (see \eqref{eq:exc7'})
$L-M \sim L-D+F \sim B+K_S+F$, where $B \sim 2E$ is an elliptic pencil with $E.L=4$. Since $E_1$ is nef, we can only have $(E_1.E,E_1.F)=(0,2)$ or $(1,0)$. The latter contradicts  \eqref{eq:exc3}, since then we would have $(E+E_1)^2=2$ and $(E+E_1).L=8$. Hence we must be in the first case with $E_1 \eqv E$. From $8 = L.(L-M)=L.(2E_1+F) = 8 + L.F$ it follows that $L.F = 0$ and $F^2 < 0$. Now there has to be a nodal curve $R < F$ such that $R.E_1 \geq 1$. Then $(2E_1+R)^2 \geq 2$ and $(2E_1+R).L=8$. The Hodge index theorem yields $(2E_1+R)^2=2$, but this again contradicts 
\eqref{eq:exc3}.

If we are in case (ii-d) then $M^2 \leq 8$. Now $M^2=6$ gives, by \eqref{eq:exc22}, the same contradiction as above, whence $M^2=8$, and consequently $(L-M)^2=2$. But this is the same case treated above, where we derived a contradiction from the fact that $h^1(L-M+K_S) \neq 0$.

\subsubsection{The case $\O_C(M)(-Z_M) \sim A-x$} 

In this case \eqref{eq:exc25} reads 
\begin{equation} 
\label{eq:exc28}
M.(L-M) + l_M + c(\F) = 8.
\end{equation} 

We now claim that 
\begin{equation} 
\label{eq:exc29}
M^2 \geq 8 \; \mbox{and} \; M.(L-M) \geq 7.
\end{equation} 

Indeed, we have $M^2 \geq 6$ by \eqref{eq:exc22}. If  $M^2 = 6$ then $h^0(M)=h^0(\I_{Z_M} \* M)$ which means that $Z_M=\emptyset$, since $M$ is base-point free by \eqref{eq:exc31}. Hence $M.L=\deg A-1=12$, whence $M.(L-M)=6$ and $(L-M)^2=8$, which is incompatible with the Hodge index theorem. Therefore $M^2 \geq 8$. Moreover, by \eqref{eq:exc23} we have $M.(L-M) \geq 6$, and equality implies $(L-M)^2 > 0$, and since $M^2+(L-M)^2=14$, we get the three possibilities $(M^2,(L-M)^2)=(12,2),(10,4),(8,6)$. The last two cases are easily ruled out using the Hodge index theorem. In the first case we get $(L-M).L=8=2\phi(L)$, contradicting \eqref{eq:exc3}. 

Therefore we have shown \eqref{eq:exc29}.  
By \eqref{eq:exc28} and \eqref{eq:exc29}  we see that the only options are
\[ ( M.(L-M), l_M, c(\F)) = (7,1,0), (7,0,1), (8,0,0). \]
In the two cases with $c(\F)=0$ we must have $M \eqv 2M_0$ with $M_0^2=2$ by Proposition \ref{prop:bassoc}(i), \eqref{eq:exc31} and \eqref{eq:exc29}. Therefore $M.(L-M)$ is even, so that we must be in the case with $M.(L-M)=8$. But then $M_0.L = 8=2\phi(L)$, contradicting \eqref{eq:exc3}. 

Therefore, the only case remaining is the one with $M.(L-M)=7$, $l_M=0$ and $c(\F)=1$. By \eqref{eq:exc19} $\F$ is good. Since $M^2 \geq 8$ by \eqref{eq:exc29} we must have $(L-M)^2 \leq 4$ by the Hodge index theorem. It is also clear that we must have $l=l_V=0$, whence twisting \eqref{eq:exc14} by $K_S$ and using \eqref{eq:exc10}, \eqref{eq:nuova} and that $h^1(\E_0(K_S))=0$ by \eqref{eq:extra} we find
\[ h^0(\F(K_S)) = h^0(\E_0(K_S)) - \chi(L-M) = h^1(A-x)- \chi(L-M) = 4- \frac{1}{2}(L-M)^2 \geq 2. \]  
Therefore we are in one of the four cases in Proposition \ref{prop:bassoc}(ii).

If we are in case (ii-a) then $M.L \geq 5\phi(L)=20$ which contradicts \eqref{eq:exc30}.

We cannot be in case (ii-b), as $\rk \F=3$.

If we are in case (ii-c) then $M^2=10$ and $\phi(M)=2$. As above we can write $M \sim 2E_1+E_2+E_3$, with $E_i>0$, $E_i^2=0$ and $E_i.E_j=1$ for $i \neq j$. Since $M.L=17$ we must have $E_i.L =4$ for at least two of the $E_i$'s, contradicting \eqref{eq:exc3}.

If we are in case (ii-d) then $M^2 \leq 8$, whence \eqref{eq:exc29} implies $M^2=8$ and \eqref{eq:exc31} implies that $\phi(M)=2$. Also $(L-M)^2=4, (L-M).L= 11$ and we claim that $h^1(L-M+K_S)=0$. Indeed, if not, by \cite[Cor.2.5]{klvan}, there exists a $\Delta > 0$ such that $\Delta^2 = - 2$ and $k:= - \Delta.(L-M) \geq 2$. Also $\Delta.M \geq 2 + \Delta.L \geq 2$. By \cite[Lemma2.3]{kl1} there exists $G>0$ such that $G^2 = 4$ and $L-M \sim G + k\Delta$. Now $11 = (L-M).L= L.G + k L.\Delta \geq L.G$ implies $\phi(G) = 2$, for, otherwise, we can write $G = 2E' + E''$ with $E' > 0, E'' > 0, (E')^2 = (E'')^2 = 0$, giving the contradiction $11 \geq L.G \geq 3 \phi(L) = 12$. Hence, as $\phi(G) = 2$, we find the contradiction $7 = M.(L-M) = M.G + k M.\Delta \geq 2 \phi(G) + 4 = 8$. Therefore $h^1(L-M+K_S)=0$ and, by \eqref{eq:exc8}, we have $h^0(\I_{Z_M} \* M)=h^0(A-x)=4$. Now $\deg (A-x)=12$, $M.L=15$ and $h^0(M)=5$, so that $Z_M$ is a scheme of length three imposing only one condition on $|M|$. Since $|M|$ is base-point free, any subscheme of length one of $Z_M$ poses one condition. Pick a subscheme $X \subset Z_M$ of length two. By the Reider-like results in \cite[Prop.3.7]{kn1} there is a $\Delta >0$ such that $X \sub \Delta$ and
$2\Delta^2 \leq M.\Delta \leq \Delta^2 +2 \leq 4$.
Also note that, since $l_M=0$, we have that $Z_M = \Bs|\I_{Z_M} \* M|$. As $Z_M$ is curvilinear of length $3$, as in the proof of Proposition \ref{prop:bassoc}, the general divisor $M_0 \in |\I_{Z_M} \* M|$ is irreducible. 

Now $\Delta^2 \leq -2$ implies $M.\Delta=0$, a contradiction, since $X$ is contained in an irreducible member $M_0$ of $|M|$. Moreover $\Delta^2=2$ implies $M \eqv 2\Delta$ by the Hodge index theorem, which is impossible since $M.(L-M)=7$. Therefore $\Delta^2=0$ and $\Delta.M=2$. It follows that $\Delta \cap M_0 = X$ and therefore $Z_M \not\subset \Delta$. By \cite[Thm.3.2.1]{cd} we can write $\Delta = \Gamma + \Delta'$ with $\Gamma \geq 0$ and $\Delta'$ of canonical type, in particular nef. But then, as $\phi(M)=2$, we must have $\Gamma.M = 0, \Delta'.M=2$, $\Delta'$ is primitive, $\Gamma \cap M_0 = \emptyset$ and again $\Delta' \cap M_0 = X$ and $Z_M \not\subset \Delta'$. 

Set $N = M - \Delta'$ and $y = Z_M - X$ (as divisors on $C$). Since any member of $|M|$ passing through a point in $\Supp Z_M$ contains the whole $Z_M$ and $Z_M \not\subset \Delta'$, we have that $y$ is a base-point of $|N|$. Now let $|N| = |N_0| + F'$ be the decomposition into the moving and fixed part. We claim that $N_0^2 = 4$ and $\phi(N_0) = 2$. In fact $N^2 = 4$ gives $h^0(N_0) = h^0(N) \geq 3$ whence either $N_0^2 = 0$ and $N_0 = kP_0$ for some $k \geq 2$ and some elliptic pencil $P_0$ or $N_0^2 > 0$. Since $N_0.L = N.L - F'.L \leq M.L - \Delta'.L \leq 15 - \phi(L) = 11 < 3 \phi(L) = 12$, the first case cannot hold, therefore $N_0^2 > 0, h^1(N_0) = 0$ and Riemann-Roch implies that $N_0^2 \geq 4$. But $0 \leq M.N_0 \leq M.(M - \Delta') \leq 8 - \phi(M) = 6$ and the Hodge index theorem yields $N_0^2 = 4$. Also $N_0.L < 3 \phi(L)$ gives $\phi(N_0) = 2$. Therefore $|N_0|$ is base-point free and $y \in \Supp F'$.

Next note that $M - F' \sim \Delta' + N_0$ is nef and $(\Delta' + N_0)^2 \geq 8, \phi(\Delta' + N_0) \geq 2$ and therefore $|M - F'|$ is base-point free. Pick any $M' \in |M - F'|$ not intersecting $Z_M$. Then $M' + F' \in |M|$ passes through $y$ and therefore it contains the whole $Z_M$. It follows that $Z_M \subset F'$ and $M.F' = M_0.F' \geq 3$. But then we get the contradiction
\[ 8 = M^2 = M.(M-F') + M.F' = M.(\Delta' + N_0) + M.F' \geq 5 + 2 \phi(M) = 9 .\]

This concludes the proof of the case $4 \leq r \leq 9$. 

\section{Curves of higher Clifford dimensions}
\label{sec:r>9}

In this section we will show that there is no exceptional curve of Clifford dimension $r \geq 10$ on an Enriques surface.

Assume to get a contradiction that, on an Enriques surface $S$, there is an exceptional curve $C$ of Clifford index $c$ and Clifford dimension $r \geq 10$ . Set $L = \O_S(C)$ and let $A$ be a line bundle on $C$ computing the Clifford dimension. Then $h^0(A)=r+1$ and, as $A$ computes the Clifford index of $C$, $A$ is base-point free and the vector bundle $\E:=\E(C,A)$ is defined.
Note that by \eqref{eq:E2a} we have $\phi(L) \geq 10$ and $L^2 \geq 100$. We need the following

\begin{lemma}   
\label{usefulr10}
Let $E$ be any effective divisor satisfying $E^2=0$ and $E.L=\phi(L)$. Then $h^0(\E (-2E) ) \geq 4$.
In particular, $\E$ is generated by its global sections off the (possibly zero) base divisor of $|\N_{C/S} -A|$. 
\end{lemma}

\begin{proof}
Since $h^1(\O_S(-2E))=0$, tensoring \eqref{eq:eca} by $\O_S(-2E)$, we see that 
$h^0(\E (-2E) ) \geq h^0((\N_{C/S}-A) (-2E))$. Now 
\[ h^0((\N_{C/S}-A) (-2E)) = h^0((\omega_C-A) (-E-(E+K_S)) \geq h^0(\omega_C-A)-2\phi(L) = h^1(A) -2\phi(L). \] 
From (\ref{eq:E2a}), (\ref{eq:E2b}) and Riemann-Roch we find 
\[ h^1(A)-2\phi(L)  =  g(C)+1-c-h^0(A)-2\phi(L) \geq \frac{1}{2}L^2+ 4 -5\phi(L) \geq \frac{1}{2}\phi(L)^2-5\phi(L)+4 \geq 4. \]
Then $h^0(\E (-2E) ) \geq 4$ and $h^0(\N_{C/S}-A) \geq 4$, that implies the last assertion by \eqref{eq:eca}. 
\end{proof}

By Lemma \ref{usefulr10} there is a nonzero section $s \in H^0(\E)$ vanishing along some member of $|2E|$, so that we can apply Lemma \ref{useful3} and we get a sequence \eqref{eq:E8} with $\F$ locally free of rank $r$, $\tau$ a torsion sheaf supported on a finite set and with $D \geq 2E$. In particular Lemma \ref{bigger0} applies. Set $M = \det \F$. Since $h^1(D+K_S) \leq 1$ by Lemma \ref{bigger0}, it follows from the fact that $M_{|C} \geq A$ 
that $h^0(M) \geq r$, that is
\begin{equation}
\label{eq:E12'}
M^2 \geq 2r-2 \geq 18. 
\end{equation}
Now let $P$ be an elliptic pencil such that $P.M = 2\phi(M)$, then $(M-P)^2 \geq 2$ and $P.(M-P) = 2\phi(M) >0$, whence there is a nontrivial effective decomposition $M \sim P + N$, with $N^2 \geq 2$. We must have
\begin{equation}
\label{eq:E12}
P.(M+D)=P.(N+D) =P.L \geq 2\phi(L) \geq c+3
\end{equation}
by (\ref{eq:E2b}). Moreover we claim that
\begin{equation}
\label{eq:E13}
N.(P+D) \geq c+1.
\end{equation}
Indeed, by Lemma \ref{cliffhanger} we must have $N.(P+D) \geq c$, and if equality occurs then $c = \Cliff (N+K_S)_{|C} = \Cliff (P+D)_{|C}$ and $h^1(N+K_S) = h^1(P+D)=0$. Since $A$ computes the Clifford dimension of $C$ we must have $h^0(N + K_S)=h^0((N + K_S)_{|C}) \geq h^0(A)=r+1$, and similarly for $h^0(P+D)$. Thus both $N^2 \geq 2r$ and $(P+D)^2 \geq 2r$. Letting $E_1 >0$ be any effective divisor such that $E_1^2=0$ and $E_1.N= \phi(N)$ and setting $P_1 =2E_1$ we find as above that there is an effective decomposition 
$N \sim P_1 + N_1$ with $N_1^2 \geq 4$. This yields, by (\ref{eq:E2b}) and Lemma \ref{cliffhanger}, that
\begin{eqnarray*}
    c & = & (P+D).N = P_1.(P+D) + N_1.(P+D) \geq (c+3 -P_1.N_1) +(c -P_1.N_1) \\ 
      & = & 2c+3-2P_1.N_1 =     2c+3-4\phi(N),
\end{eqnarray*}
whence $4\phi(N) \geq c+3$. By the Hodge index theorem
\[ 20 \phi(N)^2 \leq 20N^2 \leq (P+D)^2 N^2 \leq ((P+D).N)^2 =c^2 \leq (4\phi(N)-3)^2 < 16 \phi(N)^2, \]
a contradiction. This proves (\ref{eq:E13}).

\begin{lemma} 
\label{decmov}
The divisor $D$ has no decomposition $D \sim D_1 +D_2$ such that $h^0(D_1) \geq 2$ and $h^0(D_2) \geq 2$. In particular $-2 \leq c_2 (\F) - 2\rk \F \leq -1$.
\end{lemma}

\begin{proof}
From Lemma \ref{bigger0}, (\ref{eq:E12}) and (\ref{eq:E13}) we get
\begin{eqnarray*}
    c+2 & \geq & D.M = P.D + N.D \geq (c+3 -P.N) +(c+1 -P.N) = 2c+4-2P.N \\
        & =    & 2c+4-4\phi(M),
\end{eqnarray*}
whence 
\begin{eqnarray}
\label{eq:E14}
4\phi(M) \geq c+2, \hs & & \mbox{and if equality holds, then} \\
\nonumber & &  P.(N+D) =2\phi(L) =c+3 \hs \mbox{and} \hs N.(P+D) =c+1.
\end{eqnarray}
Now if $D \sim D_1 +D_2$ with $h^0(D_1) \geq 2$ and $h^0(D_2) \geq 2$ we have $D_i.M \geq 2\phi(M)$, whence $D.M \geq 4\phi(M) \geq c+2$. By Lemma \ref{bigger0} we must have equalities both places, whence by (\ref{eq:E14}) we have $P.(N+D) =2\phi(L) =c+3$ and $N.(P+D) =c+1$. But $2\phi(L) =c+3$ implies that $c$ is odd and $4\phi(M) = c+2$ implies that $c$ is even, a contradiction.

Finally, in view of Lemma \ref{bigger0}, assume to get a contradiction that $c_2 (\F) - 2\rk \F =0$. Then, by Lemmas \ref{cliffhanger} and \ref{useful3}, $c=\Cliff D_{|C}$ and $h^0(D_{|C})=h^0(D), h^1(D) = 0$, and, since $A$ computes the Clifford dimension, we must have  $h^0(D) \geq r+1$, whence $D^2 \geq 2r \geq 20$, and $D$ would have a decomposition into two moving classes.
\end{proof}

By Lemma \ref{decmov} we get $h^0(D-2E)=1$, whence from (\ref{eq:E8}) and Lemma \ref{usefulr10} we find
$h^0(\F(-2E)) \geq h^0(\E(-2E)) - h^0(D-2E) \geq 3$.
After saturating we find an exact sequence 
\begin{equation*}
0 \hpil \O_S(D_1) \hpil \F \hpil \F_1 \hpil \tau_1 \hpil 0,
\end{equation*}
where $D_1 \geq 2E$, $\F_1$ is locally free of rank $r-1 \geq 9$, is globally generated off a finite set 
and $\tau_1$ is a torsion sheaf supported on a finite set. From Lemma \ref{useful3} we find $h^0(\F_1^*)=h^1 (\F_1^*)=0$. 

As in section \ref{sec:r=3} page \pageref{port}, $M_1:= \det \F_1$ is nontrivial and base-component free (in particular it is nef and $h^0(M_1) \geq 2$) and $M \sim M_1+D_1$. We have
\[ -1 \geq c_2 (\F) - 2\rk \F = D_1.M_1 -2 + c_2 (\F_1) - 2\rk \F_1 + \length(\tau_1), \]
whence $D_1.M_1 + c_2 (\F_1) - 2\rk \F_1 \leq 1$.

If $M_1^2=0$ then $M_1 \sim mP_1$ for an elliptic pencil $|P_1|$ and an integer $m \geq 1$. By \cite[Prop.3.2]{kn2} we have $c_2 (\F_1) - 2\rk \F_1 \geq -2m$, whence $m(P_1.D_1-2) \leq 1$ so that we must have $P_1.D_1=0$ or $2$ (recalling that $P_1.D_1$ is even). Since $P_1.D_1=P_1.M$ and $M^2 >0$ we must have $P_1.D_1=2$. It easily follows that, as $P_1 = 2G$ for some $G$, we can write $M \sim (\frac{1}{4}M^2)P_1+E_1$ for some $E_1 >0$ satisfying $E_1^2=0$. Since $D^2 >0$ by Lemma \ref{bigger0}, we have $E_1.D \geq 1$ and, by (\ref{eq:E2b}), $P_1.D=P_1.L-P_1.M \geq 2\phi(L)-2 \geq c+1$, whence by (\ref{eq:E12'})
\[ M.D= (\frac{1}{4}M^2)P_1.D+E_1.D \geq \frac{9}{2}(c+1)+1 > 4c+5, \]
contradicting the fact that $M.D \leq c+2$ by Lemma \ref{bigger0}.

Hence $M_1^2 >0$, so that by \cite[Prop.3.2]{kn2} we have $c_2 (\F_1) - 2\rk \F_1 \geq -2$, whence 
\begin{equation}
\label{eq:E1233}
D_1.M_1 \leq 3
\end{equation}
and, by the Hodge index theorem, $D_1^2M_1^2 \leq 9$.
This gives $D_1^2 \leq 4$ whence, by (\ref{eq:E12'}) and (\ref{eq:E1233}), $18 \leq M^2 = M_1^2+D_1^2+2M_1.D_1 \leq M_1^2 + 10$, giving $M_1^2 \geq 8$.
As $D_1 \geq 2E$ and $M_1$ is nef, we deduce from (\ref{eq:E1233}) that $E.M_1 = 1$ and we can write $M_1 = (\frac{1}{2}M_1^2) E + F$ for some $F > 0$ with $F^2 = 0, E.F=1$.
By Lemma \ref{cliffhanger} we have $D_1.(M_1 + D) \geq c$, whence by (\ref{eq:E1233}), $D_1.D \geq c - 3$. Now, using $M_1 = (\frac{1}{2}M_1^2) E + F$, $M_1^2 \geq 8$ and $\phi(L) \geq 10$, we get
\[ M_1.D = M_1.L - M_1.M = M_1.L - M_1^2 - M_1.D_1 \geq \frac{1}{2}M_1^2 (\phi(L) -2) + \phi(L) - 3 \geq 39 \]
whence Lemma \ref{bigger0} gives the contradiction
\[ c + 2 \geq D.M = D.M_1 + D.D_1 \geq 36 + c. \]

This concludes the proof of the case $r \geq 10$ and therefore also the proof of Theorem \ref{thm:noexc}.

\section{$W^1_4$ on curves of genus $9$}
\label{sec:w14}

In this section we extend to the case $L^2=16$ the result in \cite[Prop.4.15]{kl1}.

\begin{prop} 
\label{casi14,16}
Let $L$ be a base-point free line bundle on an Enriques surface with $L^2 = 16$ and 
$\phi(L) = 2$. Let $C$ be a general smooth curve in $|L|$. Then $W^1_4(C)$ is smooth. 
\end{prop}

\begin{proof}
By \cite[Prop.4.15]{kl1} there is a unique genus one pencil $|2E|$ on $S$ such that $E.L=2$ and there is a unique $g^1_4$ on $C$, 
namely $A = (2E)_{|C}$. As in the proof of  \cite[Prop.4.15]{kl1} we need to prove that the multiplication map 
$\mu_{0, A} : H^0(A) \otimes H^0(K_C-A) \to H^0(K_C)$ is surjective. If not,
by the base-point-free pencil trick, we have $h^0(K_C - 2A) \geq 4$. 
Since $\deg(K_C - 2A) = 8$ and $\Cliff(C) = 2$ by Corollary \ref{cliff} we see that $|K_C - 2A|$ 
is a base-point free $g^3_8$ on $C$. 
If $|K_C - 2A|$ is not very ample, there are two points $P, Q \in C$ such that
$|K_C - 2A - P - Q|$ is a $g^2_6$, that must be base-point free (since $\Cliff(C) = 2$), contradicting Proposition \ref{g^2_6particolari}.
Therefore $|K_C - 2A|$ is very ample and $C$ is isomorphic to a smooth curve of degree
$8$ and genus $9$ in $\PP^3$, hence to a smooth complete intersection of an irreducible quadric and
a quartic surface in $\PP^3$. Hence $C$, in its canonical embedding, is isomorphic to a quadric
section of a Del Pezzo surface, contradicting \cite[Prop.5.14]{klGM}. 
\end{proof}

\end{document}